\documentclass[11pt]{article}

\usepackage{amssymb}
\usepackage{epsfig}
\usepackage{mathdots}
\usepackage{amsthm}
\usepackage{amsmath}

\newtheorem{theorem}{Theorem}
\newtheorem{corollary}[theorem]{Corollary}

\newtheorem{lemma}[theorem]{Lemma}
\newtheorem{proposition}[theorem]{Proposition}

\theoremstyle{definition}
\newtheorem{example}[theorem]{Example}
\newtheorem{definition}[theorem]{Definition}
\newtheorem{remark}[theorem]{Remark}

\numberwithin{equation}{section}
\numberwithin{theorem}{section}

\newcommand{\bT}{\begin{theorem}}
\newcommand{\eT}{\end{theorem}}
\newcommand{\bProp}{\begin{proposition}}
\newcommand{\eProp}{\end{proposition}}
\newcommand{\bE}{\begin{example}}
\newcommand{\eE}{\end{example}}
\newcommand{\bL}{\begin{lemma}}
\newcommand{\eL}{\end{lemma}}
\newcommand{\bP}{\begin{proof}}
\newcommand{\eP}{\end{proof}}
\newcommand{\bC}{\begin{corollary}}
\newcommand{\eC}{\end{corollary}}
\newcommand{\bD}{\begin{definition}}
\newcommand{\eD}{\end{definition}}
\newcommand{\be}{\begin{enumerate}}
\newcommand{\ee}{\end{enumerate}}
\newcommand{\beqa}{\begin{eqnarray*}}
\newcommand{\eeqa}{\end{eqnarray*}}
\newcommand{\beqaa}{\begin{eqnarray}}
\newcommand{\eeqaa}{\end{eqnarray}}
\newcommand{\ba}{\begin{array}}
\newcommand{\ea}{\end{array}}

\title{Some implications of Chu's $_{10}\psi_{10}$
extension of \\Bailey's $_{6}\psi_{6}$ summation formula}

\author{James McLaughlin, Andrew V. Sills, Peter Zimmer}

\date{\today}

\begin{document}
\maketitle

key words: {$q$-Series, Rogers-Ramanujan Type Identities, Bailey chains, False Theta Series}

subject class[2000]: {Primary: 33D15. Secondary:11B65, 05A19.}

\begin{abstract}
Lucy Slater used Bailey's $_6\psi_6$ summation formula to derive the
Bailey pairs she used to construct her famous list of 130 identities
of the Rogers-Ramanujan type.

In the present paper we apply the same techniques to Chu's
$_{10}\psi_{10}$ generalization of Bailey's formula to produce quite
general Bailey pairs. Slater's Bailey pairs are then recovered as
special limiting cases of these more general pairs.

In re-examining Slater's work, we find that her Bailey pairs are,
for the most part, special cases of more general Bailey pairs
containing one or more free parameters. Further, we also find new
general  Bailey pairs (containing one or more free parameters) which
are also implied by the $_6\psi_6$ summation formula.

Slater used the Jacobi triple product identity (sometimes coupled
with the quintuple product identity) to derive her
infinite products. Here we also use other summation formulae
(including special cases of the $_6\psi_6$ summation formula and Jackson's
$_6\phi_5$ summation formula) to derive some of our infinite products.

We use the new Bailey pairs, and/or the summation methods mentioned
above, to give new proofs of some general series-product identities
due to Ramanujan, Andrews and others. We also derive a new general
series-product identity, one which may be regarded as a partner to
one of the  Ramanujan identities. We also find new transformation
formulae between basic hypergeometric series, new identities of
Rogers-Ramanujan type, and new false theta series identities. Some of these latter are a kind of ``hybrid" in that one side of the identity consists a basic hypergeometric series, while the other side is formed from a theta product multiplied by a false theta series. This type of identity appears to be new.
\end{abstract}

\section{Introduction}
Bailey's  $_{6}\psi_6$ identity \cite{B36},
\begin{equation}\label{baileyeq1}
_{6}\psi_{6}
\left [
\begin{matrix}
q\sqrt{a},-q\sqrt{a},b,c,d,e\\
\sqrt{a},-\sqrt{a},qa/b,qa/c,qa/d,qa/e
\end{matrix};q,\frac{q a^2}{bcde}
\right ]
=\frac{ (aq,aq/bc,aq/bd,aq/be,aq/cd,aq/ce,aq/de,q,q/a;q)_{\infty} } {
(aq/b,aq/c,aq/d,aq/e,q/b,q/c,q/d,q/e,qa^2/bcde;q)_{\infty} },
\end{equation}
is probably the most important summation formula for  bilateral basic hypergeometric series,
 and has many applications to number theory and partitions - see Andrews' paper \cite{A74}
 for some examples. This summation formula was also the main tool used by Slater
 \cite{S51, S52} to derive Bailey pairs and, from these,
  her list of 130 identities of the Rogers-Ramanujan type.

Bailey's formula was extended by Shukla \cite{S59}, and Shukla's
formula was later further extended by Chu \cite{C05}. Let
{\allowdisplaybreaks
\begin{multline}\label{Keq}
K:=K(a,b,c,d,e,u,v,q)
=\big [
u v \left(\sigma _3-a \sigma _1\right) \left(q \sigma _3-a \sigma _1\right) a^2
+(1-q) u v \left(a^2-\sigma _4\right)
   \left(a^2-\sigma _2 a+\sigma _4\right) a\\+(u+v) (a+u v) \left(\sigma _3-a \sigma _1\right) \left(a^2-q \sigma _4\right)
   a
   +\left(u^2+a\right) \left(v^2+a\right) \left(a^2-\sigma _4\right) \left(a^2-q \sigma _4\right)
\big ]\\
\times \{\left(a^2-b c d e\right) \left(a^2-b c d e q\right) (1-u) (a-u) (1-v) (a-v)\}^{-1},
\end{multline}
} where, for $1\leq k \leq 4$, $\sigma_k$ denotes the $k$-th
elementary symmetric function in $\{b,c,d,e\}$. Then

\begin{proposition}$($Chu \cite{C05}$)$ For complex numbers $a$, $b$, $c$, $d$ and $e$ satisfying $|a^2/bcdeq|<1$, there holds the identity
{\allowdisplaybreaks
\begin{multline}\label{chueq1}
_{10}\psi_{10}
\left [
\begin{matrix}
q\sqrt{a},-q\sqrt{a},b,c,d,e,qu,qa/u,qv,qa/v\\
\sqrt{a},-\sqrt{a},qa/b,qa/c,qa/d,qa/e,u,a/u,v,a/v
\end{matrix};q,\frac{a^2}{qbcde}
\right ]
\\
= \sum_{n=0}^{\infty}\frac{(1-a
q^{2n})(b,c,d,e,uq,vq,aq/u,aq/v;q)_n}
{(1-a)(aq/b,aq/c,aq/d,aq/e,u,v,a/u,a/v;q)_n} \left(
\frac{a^2}{qbcde}\right)^n \phantom{adsadasdadsasdasdsdad}\\
+ \sum_{n=1}^{\infty}\frac{(1-
q^{2n}/a)(b/a,c/a,d/a,e/a,uq/a,vq/a,q/u,q/v;q)_n}
{(1-1/a)(q/b,q/c,q/d,q/e,1/u,1/v,u/a,v/a;q)_n} \left(
\frac{a^2}{qbcde}\right)^n \\
=K \, \frac{
(aq,aq/bc,aq/bd,aq/be,aq/cd,aq/ce,aq/de,q,q/a;q)_{\infty} } {
(aq/b,aq/c,aq/d,aq/e,q/b,q/c,q/d,q/e,qa^2/bcde;q)_{\infty} }.
\end{multline}
}
\end{proposition}
Upon letting $u\to \infty$, $v\to \infty$, we obtain Bailey's  formula \eqref{baileyeq1},
while letting $u \to \infty$ recovers Shukla's \cite{S59} $_8\psi_8$ generalization of Bailey's formula. For most values of the parameters, the expression for $K$ is quite complicated, but we note in passing that the case $b=a/c$ results in considerable simplification.

\begin{corollary} For complex numbers $a$, $c$, $d$ and $e$ satisfying $|a/deq|<1$, there holds the identity
{\allowdisplaybreaks
\begin{multline}\label{chueq2}
_{10}\psi_{10}
\left [
\begin{matrix}
q\sqrt{a},-q\sqrt{a},a/c,c,d,e,qu,qa/u,qv,qa/v\\
\sqrt{a},-\sqrt{a},cq,qa/c,qa/d,qa/e,u,a/u,v,a/v
\end{matrix};q,\frac{a}{qde}
\right ]\\
=\frac{(1 - u/c)(1 - cu/a)(1 - v/c)(1 - cv/a)}
{(1 - u/a)(1 - u)(1 - v/a)(1 - v)}
\frac{ (aq,q,cq/d,cq/e,aq/cd,aq/ce,aq/de,q,q/a;q)_{\infty} } {
(cq,aq/c,aq/d,aq/e,cq/a,q/c,q/d,q/e,qa/de;q)_{\infty} }.
\end{multline}
}
\end{corollary}

Since \eqref{baileyeq1} was used by Slater to derive her Bailey
pairs, it is natural to ask if Chu's generalization of Bailey's formula at
\eqref{chueq1}  can be used similarly  to produce
any new interesting results. The present paper is in part an
investigation of that question.

One observation we make is that most of Slater's Bailey pairs are
special cases of more general Bailey pairs containing one or more
free parameters - see Corollaries \ref{csgen1} and \ref{csgen2}.
Slater could have derived these more general pairs herself, but it
would seem that she was primarily interested in the special cases
which would lead to identities of the Rogers-Ramanujan type.

We also note that Slater used the Jacobi Triple Product identity to
derive her infinite products. In the present paper we use other
summation formulae, including special cases of the $_6\psi_6$-
 and  $_6\phi_5$ summation formulae, to derive some infinite
 products.

 In  the transformation at \eqref{Seq1}
 below, Slater essentially employed three cases ($y,z \to \infty$
 and $y = \pm \sqrt{aq}$, $z \to \infty$) to make the series
 containing the $\alpha_n$ sequences summable. In the present paper
 we also explore additional cases, in the process discovering some
interesting  new identities (see Section \ref{secids} below).

Our results include new proofs of some general series-product
identities due to Ramanujan, Andrews and others, and also a new
general series-product identity
\begin{equation}
\sum_{n=0}^\infty \frac{(q/z;q)_{n+1} (z;q)_n q^{n^2+n} }{(q;q)_{2n+1} }
= \frac{ (zq^2, q/z, q^3;q^3)_\infty}{(q;q)_\infty},
\end{equation}
which may be regarded as a
partner to an identity equivalent to one recorded by Ramanujan in his lost
notebook~\cite[p. 99, Entry 5.3.1 with $a=-z/q$ and $q\to\sqrt{q}$
throughout]{AB09}:
\begin{equation}
 \sum_{n=0}^\infty \frac{(q/z;q)_{n} (z;q)_n q^{n^2} }{(q;q)_{2n} }
= \frac{ (zq, q^2/z, q^3;q^3)_\infty}{(q;q)_\infty}.
\end{equation}
We also find new
transformation formulae between basic hypergeometric series, new
identities of Rogers-Ramanujan type, and new false theta series
identities. Some of these latter are a kind of ``hybrid" in that one side of the identity consists a basic hypergeometric series, while the other side is formed from a theta product multiplied by a false theta series.  For example,
\begin{equation}
  \sum_{n=0}^\infty \frac{ q^{n^2+n} }{(1-q^{2n+1}) (q;q)_n^2 }
  = \frac{ \sum_{n=0}^\infty q^{6n^2+2r} (1-q^{8r+4})}{(q;q)_\infty}
\end{equation}

This type of identity appears to be new.

We employ the usual notations:
\begin{align*}
           (a;q)_n &:= (1-a)(1-aq)\cdots (1-aq^{n-1}), \\
          (a_1, a_2, \dots, a_j; q)_n &:= (a_1;q)_n (a_2;q)_n \cdots (a_j;q)_n ,\\
           (a;q)_\infty &:= (1-a)(1-aq)(1-aq^2)\cdots, \mbox{ and }\\
          (a_1, a_2, \dots, a_j; q)_\infty &:= (a_1;q)_\infty (a_2;q)_\infty \cdots (a_j;q)_\infty,
\end{align*}

For later use, we recall that a  pair of sequences $(\alpha_n, \beta_n)$ that satisfy  $\alpha_0=1$ and
\begin{equation}\label{bpeq}
\beta_n = \sum_{j=0}^{n} \frac{\alpha_j}{(q;q)_{n-j}(aq;q)_{n+j}},
\end{equation}
is termed a \emph{Bailey pair relative to $a$.}

\section{Bailey Pairs from Chu's Extension of the $_6\psi_6$ summation formula}
\subsection{General Bailey pairs}
We first consider the case $e=a$ and $d=q^{-N}$ in Chu's formula. These same substitutions were made by Slater in \eqref{baileyeq1}, and gave rise to a quite general Bailey pair (see below), which in turn led to many new identities of Rogers-Ramanujan type.
Let $K$ be defined as at \eqref{Keq}.
\begin{theorem}\label{t2}
The pair of sequences $(\alpha_n, \beta_n)$ is a Bailey pair with respect to $a$,
where
\begin{align}\label{chubp1}
\alpha_n&:=
\frac{(q\sqrt{a},-q\sqrt{a},a,b,c,qu,qa/u,qv,qa/v;q)_n}
{(\sqrt{a},-\sqrt{a},qa/b,qa/c,u,a/u,v,a/v,q;q)_n}
\left( \frac{-a}{bc}\right)^n q^{n(n-3)/2}, \\
\beta_n&:= K_1(n) \frac{(aq/bc;q)_n}{(aq/b,aq/c,q;q)_n},\notag
\end{align}
where $K_1(n):=K(a,b,c,q^{-n},a,u,v,q)$.
\end{theorem}
\begin{proof}
First set $e=a$, so that all the terms with negative index in the
bilateral sum at \eqref{chueq1} become zero. Next, for each
non-negative integer $n$, set $d=q^{-n}$, so that the sum on the
left side of \eqref{chueq1} becomes a finite sum, with the summation
index running from 0 to $n$. Simplify the resulting product on the
right side of \eqref{chueq1}, and use the identity (see
\cite[(I.10), page 351]{GR04})
\begin{equation}\label{q-njeq}
(q^{-n};q)_j=\frac{(q;q)_n q^{j(j-1)/2}}{(q;q)_{n-j}(-q^n)^j}
\end{equation}
to modify the series side. The result then follows from \eqref{bpeq}, after some simple manipulations, noting also that $K_1(0)=1$, so that the requirement $\alpha_0=\beta_0=1$ is satisfied.
\end{proof}

The expression for $K_1(n)$ above is  generally
quite complicated, and thus so also is the expression for the $\beta_n$.
However, as was also the case for $K$ above, if we set
$b=a/c$ then $K_1(n)$ simplifies considerably. One can check, preferably
using a computer algebra system, that
{\allowdisplaybreaks
\begin{equation}
K(a,a/c,c,q^{-n},a,u,v,q)=
\begin{cases}1, &n=0,\\
&\\
\displaystyle{\frac{(c - u)(c - v)(a - c u)(a - c v)} {c^2(a - u)(a
- v)(1 - u)(1 - v)}} &\phantom{a}\\
-\displaystyle{\frac{(1-c) (c-a) (c-a q) (1-c q) u v}
{c^2 q (a-u)
   (1-u) (a-v) (1-v)}}, &n=1,\\
   &\\
\displaystyle{\frac{(c - u)(c - v)(a -
c u)(a - c v)} {c^2(a - u)(a - v)(1 - u)(1 - v)}}, &n>1.
\end{cases}
\end{equation}
}

Upon letting $u\to \infty$, $v \to \infty$ in the Bailey pair in
Theorem \ref{t2}, we recover the following Bailey pair (with respect to
$a$) due to Slater,
{\allowdisplaybreaks
\begin{align}\label{Sgen1}
\alpha_n&=\frac{(1-a q^{2n})(a,b,c;q)_n}{(1-a)(aq/b,aq/c,
q;q)}\left( \frac{-a}{bc}\right)^n q^{(n^2+n)/2},\\
\beta_n&=\frac{(aq/bc;q)_n}{(aq/b,aq/c,q;q)_n}, \notag
\end{align}
}
which is implicitly contained in Equation (4.1) of \cite{S51}.

It would appear that Slater's principal motivation was to prove identities of the Rogers-Ramanujan type, so  this Bailey pair, and other general pairs mentioned below, were not stated explicitly by her in
\cite{S51} and \cite{S52}, where she instead listed many special
cases of them. However, they could all have been easily derived by her, using the same methods she used to derive the special cases.

We remark in passing that all the Bailey pairs in Slater's \textbf{B}, \textbf{F} and \textbf{H}
tables, as well as pairs \textbf{E(3)}, \textbf{E(6)} and
\textbf{E(7)} (see \cite[page 468]{S51}), are derived from the
 Bailey pair at \eqref{Sgen1}.

Slater also showed \cite[Equation (1.3) on page 462]{S51} that if $(\alpha_n, \beta_n)$ is a Bailey pair with respect to $a$, then, for non-zero complex numbers $y$ and $z$,
\begin{equation}\label{Seq1}
\sum_{n=0}^{\infty}(y,z;q)_n \left (\frac{aq}{yz}\right)^n \beta_n
=\frac{(aq/y,aq/z;q)_{\infty}}{(aq,aq/yz;q)_{\infty}}
\sum_{n=0}^{\infty}\frac{(y,z;q)_n}{(aq/y,aq/z;q)_n}\left (\frac{aq}{yz}\right)^n \alpha_n.
\end{equation}

A finite generalization of this identity is of course implied by Bailey's Lemma (see, for example, Theorem 12.2.3 in \cite{AAR99}), namely, if  $(\alpha_n, \beta_n)$ is a Bailey pair with respect to $a$, and $N$ is a non-negative integer, then
{\allowdisplaybreaks
\begin{equation}\label{wpbteq1}
\sum_{n=0}^{N} \frac{(y,z,q^{-N};q)_n }{(y z q^{-N}/a;q)_n}\,q^n \beta_n=
\frac{(a q/y,a q/z;q)_{N}} {(a
q,a q/y z;q)_{N}}
\sum_{n=0}^{N} \frac{(y,z, q^{-N};q)_{n}} {(a q/y,a q/z,a q^{1+N};q)_{n}}\left ( \frac{-a q^N}{yz}\right)^n q^{-n(n-3)/2}\alpha_n.
\end{equation}
}

The identity at \eqref{Seq1} is a particular case of the \emph{Bailey
Transform}: if
$\beta_n=\sum_{r=0}^n \alpha_r u_{n-r} v_{n+r}$
and
$\gamma_n =$ $ \sum_{r=n}^\infty \delta_r u_{r-n}$ $v_{r+n}$,
then
\[ \sum_{n=0}^\infty \alpha_n \gamma_n = \sum_{n=0}^\infty  \beta_n \delta_n. \]
In the present paper, for ease of notation we will refer to
\eqref{Seq1} as the \emph{Bailey Transform}.

If we set $b=a/c$ in the Bailey pair from Theorem \ref{t2},
 and then substitute this pair into \eqref{Seq1} and \eqref{wpbteq1}, we get the following unusual basic hypergeometric identities.
\begin{corollary}\label{C65}
Let  $N\geq1$ be an integer. Then
\begin{multline}\label{wpbteqcase1}
1+\frac{(c - u)(c - v)(a - c u)(a - c v)} {c^2(a - u)(a
- v)(1 - u)(1 - v)}\sum_{n=1}^{N} \frac{(y,z,q^{-N};q)_n }
{\left(c q,\displaystyle{\frac{a q}{c},\frac{y z q^{-N}}{a}};q\right)_n}\,q^n
=\frac{a  (c-a) (1-c) (1-y)
   (1-z)\left(1-q^N\right) u v}{c (a-u)  (a-v) (1-u)(1-v) \left(y z-a q^N\right)}\\
   +
\frac{\left(\displaystyle{\frac{a q}{y},\frac{a q}{z}};q\right)_{N}} {\left(a
q,\displaystyle{\frac{a q}{y z}};q\right)_{N}}
\sum_{n=0}^{N} \frac{\left(q\sqrt{a},-q\sqrt{a},a,
\displaystyle{\frac{a}{c},c,qu,\frac{qa}{u},qv,\frac{qa}{v}},y,z, q^{-N};q\right)_{n}}
 {\left(\sqrt{a},-\sqrt{a},c q,\displaystyle{\frac{qa}{c},u,\frac{a}{u},v,\frac{a}{v},\frac{a q}{y},\frac{a q}{z}},a q^{1+N},q;q\right)_{n}}
 \left ( \frac{a q^N}{yz}\right)^n;
\end{multline}
\begin{multline}\label{Seqcase1}
1+\frac{(c - u)(c - v)(a - c u)(a - c v)} {c^2(a - u)(a
- v)(1 - u)(1 - v)}\sum_{n=1}^{\infty}\frac{(y,z;q)_n}{\left(c q,\frac{a q}{c};q\right)_n} \left (\frac{aq}{yz}\right)^n
=\frac{a  (c-a) (1-c) (1-y)
   (1-z) u v}{c (a-u)  (a-v) (1-u)(1-v) y z}\\
   +
\frac{\left(\displaystyle{\frac{a q}{y},\frac{a q}{z}};q\right)_{\infty}} {\left(a
q,\displaystyle{\frac{a q}{y z}};q\right)_{\infty}}
\sum_{n=0}^{\infty} \frac{\left(q\sqrt{a},-q\sqrt{a},a,
\displaystyle{\frac{a}{c},c,qu,\frac{qa}{u},qv,\frac{qa}{v}},y,z;q\right)_{n}}
 {\left(\sqrt{a},-\sqrt{a},c q,\displaystyle{\frac{qa}{c},u,\frac{a}{u},v,\frac{a}{v},\frac{a q}{y},
 \frac{a q}{z}},q;q\right)_{n}}
\left (\frac{-a}{yz}\right)^n q^{n(n-1)/2}.
\end{multline}
\end{corollary}
Note that setting $c=a$ in \eqref{wpbteqcase1} gives the
$q$-Pfaff-Saalsch\"{u}tz sum (see \cite[page 355, II.12]{GR04}),
while setting $c=u$  gives the identity (for $N\geq 0$)
{\allowdisplaybreaks
\begin{equation}\label{wpbteqcase1a}
\sum_{n=0}^{N} \frac{\left(q\sqrt{a},-q\sqrt{a},a,
\displaystyle{qv,\frac{qa}{v}},y,z, q^{-N};q\right)_{n}}
 {\left(\sqrt{a},-\sqrt{a},v,\displaystyle{\frac{a}{v},\frac{a q}{y},\frac{a q}{z}},a q^{1+N},q;q\right)_{n}}
 \left ( \frac{a q^N}{yz}\right)^n
=\left [ 1 + \frac{a   (1-y)
   (1-z)\left(1-q^N\right)  v}{  (a-v) (1-v) \left(y z-a q^N\right)} \right ]
\frac{\left(a
q,\displaystyle{\frac{a q}{y z}};q\right)_{N}}{\left(\displaystyle{\frac{a q}{y},\frac{a q}{z}};q\right)_{N}}.
\end{equation}
}

\subsection{ Mod $3$ Bailey pairs}

We continue to follow in Slater's footsteps,  this time making the same substitutions in
Chu's identity \eqref{chueq1} that she did in \eqref{baileyeq1} to produce the Bailey pairs in her \textbf{A} table. Let $K$ be  as defined at \eqref{Keq}.
\begin{theorem}\label{t23}
(i) Let $K_0(n):=K(a,q^{-n},q^{1-n},q^{2-n},a,u,v,q^3)$. Then the
pair of sequences $(\alpha_n, \beta_n)$ is a Bailey pair with
respect to $a=a$, where $\alpha_0=\beta_0=1$, $\alpha_{3r\pm1}=0$,
and {\allowdisplaybreaks
\begin{align}\label{chubp22}
\alpha_{3r}&=  \frac{1-a q^{6r}}{1-a}
\frac{\left(\displaystyle{a,q^3u,q^3v,a q^3/u,a q^3/v;q^3}\right)_r}
   {\left(\displaystyle{u,v,a/u,a/v,q^3;q^3}\right)_r}q^{\frac{9r^2-15r}{2}}(-a)^r\\
\beta_n&= K_0(n) \frac{(a ;q^3)_{n}}{(a;q)_{2n}(q;q)_n}.\notag
\end{align}
} Let $K_0'(n):=K(aq,q^{-n},q^{1-n},q^{2-n},aq,u,v,q^3)$. Then\\
(ii)  the pair of sequences $(\alpha_n, \beta_n)$ is a Bailey pair
with respect to $a=a$, where $\alpha_0=\beta_0=1$,
$\alpha_{3r-1}=0$, and {\allowdisplaybreaks
\begin{align}\label{chubp22a}
\alpha_{3r}&= \frac{\left(\displaystyle{a q,q^3u,q^3v,a q^4/u,a
q^4/v;q^3}\right)_r}
   {\left(\displaystyle{u,v,a q/u,a q/v,q^3;q^3}\right)_r}q^{\frac{9r^2-13r}{2}}(-a)^r\\
\alpha_{3r+1}&= \frac{\left(\displaystyle{a q,q^3u,q^3v,a q^4/u,a
q^4/v;q^3}\right)_r}
   {\left(\displaystyle{u,v,a q/u,a q/v,q^3;q^3}\right)_r}q^{\frac{9r^2-r}{2}+1}(-a)^{r+1}\notag\\
   \beta_n&= K_0'(n) \frac{(aq ;q^3)_{n}}{(aq;q)_{2n}(q;q)_n};\notag
\end{align}
} (iii)  the pair of sequences $(\alpha_n, \beta_n)$ is a Bailey
pair with respect to $a=a$, where $\alpha_0=\beta_0=1$,
$\alpha_{3r-1}=0$, and {\allowdisplaybreaks
\begin{align}\label{chubp22b}
\alpha_{3r}&= \frac{\left(\displaystyle{a q,q^3u,q^3v,a q^4/u,a
q^4/v;q^3}\right)_r}
   {\left(\displaystyle{u,v,a q/u,a q/v,q^3;q^3}\right)_r}q^{\frac{9r^2-7r}{2}}(-a)^r\\
\alpha_{3r+1}&= -\frac{\left(\displaystyle{a q,q^3u,q^3v,a q^4/u,a
q^4/v;q^3}\right)_r}
   {\left(\displaystyle{u,v,a q/u,a q/v,q^3;q^3}\right)_r}q^{\frac{9r^2-7r}{2}}(-a)^{r}\notag\\
   \beta_n&= K_0'(n) \frac{(aq ;q^3)_{n}q^n}{(aq;q)_{2n}(q;q)_n}.\notag
\end{align}
}
\end{theorem}
\begin{proof}
Set $e=a$ in \eqref{chueq1}, so that all the terms of negative
index vanish. Then replace $q$ with $q^3$, set $b=q^{-n}$,
$c=q^{1-n}$ and $d=q^{2-n}$.  Then after some simple
manipulations  \eqref{chueq1} becomes
\begin{equation*}
\sum_{r=0}^{n/3}\frac{1-a q^{6r}}{1-a}
\frac{(q^{-n};q)_{3r}}{(a q^{n+1};q)_{3r}}
\frac{\left(\displaystyle{a,q^3u,q^3v,aq^3/u,aq^3/v;q^3}\right)_r}
{\left(\displaystyle{u,v,a/u,a/v,q^3;q^3}\right)_r}\left(a q^{3n-6}
\right)^r
= K_0(n) \frac{(a;q^3)_{n}(a q;q)_n}{(a;q)_{2n}}.
\end{equation*}
Apply \eqref{q-njeq} to the $(q^{-n};q)_{3r}$ factor, divide both
sides by $(aq;q)_n(q;q)_n$ to get
{\allowdisplaybreaks
\begin{equation}\label{t3eq}
\sum_{r=0}^{n/3}\frac{1-a q^{6r}}{1-a} \frac{(-1)^r
q^{(9r^2-15r)/2}}{(a q;q)_{n+3r}(q;q)_{n-3r}}
\frac{\left(\displaystyle{a,q^3u,q^3v,aq^3/u,aq^3/v;q^3}\right)_r}
{\left(\displaystyle{u,v,a/u,a/v,q^3;q^3}\right)_r}a^r
= K_0(n) \frac{(a;q^3)_{n}}{(a;q)_{2n}(q;q)_n}
\end{equation}
}and \eqref{chubp22} follows.

For the other two pairs, replace $a$ with $aq$ in \eqref{t3eq}, and
\eqref{chubp22a} follows from the identity
\begin{equation*}
\frac{(1-a q^{6r+1}) q^{(9r^2-13r)/2}}{(a
q;q)_{n+3r+1}(q;q)_{n-3r}} =\frac{ q^{(9r^2-13r)/2}}{(a
q;q)_{n+3r}(q;q)_{n-3r}} -\frac{ a\,q^{(9r^2-r)/2+1}}{(a
q;q)_{n+3r+1}(q;q)_{n-3r-1}},
\end{equation*}
while \eqref{chubp22b} follows from the identity
\begin{equation*}
\frac{(1-a q^{6r+1}) q^{(9r^2-13r)/2}}{(a q;q)_{n+3r+1}(q;q)_{n-3r}}
=\frac{q^{-n} q^{(9r^2-7r)/2}}{(a q;q)_{n+3r}(q;q)_{n-3r}} -\frac{
q^{-n}q^{(9r^2-7r)/2}}{(a q;q)_{n+3r+1}(q;q)_{n-3r-1}}.
\end{equation*}
\end{proof}

\begin{remark} $K_{0}(0)=K_{0}(1)=K_{0}(2)=1$. \end{remark}

\begin{theorem}\label{t3}
Let $K_2(n):=K(q,q^{-n},q^{1-n},q^{2-n},e,u,v,q^3)$ and suppose $e
\not = 1/q$. Then\\
(i) the pair of sequences $(\alpha_n, \beta_n)$ is a Bailey pair
with respect to $a=1$, where $\alpha_0=\beta_0=1$ and
{\allowdisplaybreaks
\begin{align}\label{chubp3}
\alpha_{3r}&=
(-1)^r \bigg(q^{\frac{9r^2-11r}{2}}
\frac{\left(\displaystyle{e,q^3u,q^3v,q^4/u,q^4/v;q^3}\right)_r}
   {\left(\displaystyle{q^4/e,u,v,q/u,q/v;q^3}\right)_r\,\displaystyle{e^r}}
+q^{\frac{9r^2-5r}{2}}
\frac{\left(\displaystyle{e/q,q^2u,q^2v,q^3/u,q^3/v;q^3}\right)_r}
   {\left(\displaystyle{q^3/e,u/q,v/q,1/u,1/v;q^3}\right)_r\,\displaystyle{e^r}}\bigg),\\
\alpha_{3r+1}&=(-1)^{r+1} q^{\frac{9r^2+r}{2}+1}\frac{\left(\displaystyle{e,q^3u,q^3v,q^4/u,q^4/v;q^3}\right)_r}
   {\left(\displaystyle{q^4/e,u,v,q/u,q/v;q^3}\right)_r\displaystyle{e^r}},\notag\\
\alpha_{3r-1}&=(-1)^{r+1}q^{\frac{9r^2-17r}{2}+1}
\frac{\left(\displaystyle{e/q,q^2u,q^2v,q^3/u,q^3/v;q^3}\right)_r}
{\left(\displaystyle{q^3/e,u/q,v/q,1/u,1/v;q^3}\right)_r\displaystyle{e^r}},\notag
\\
\beta_n&= K_2(n) \frac{(q^2/e;q^3)_n}{(q;q)_{2n}(q^2/e;q)_n};\notag
\end{align}
} (ii) the pair of sequences $(\alpha_n, \beta_n)$ is a Bailey pair
with respect to $a=1$, where $\alpha_0=\beta_0=1$ and
{\allowdisplaybreaks
\begin{align}\label{chubp5}
\alpha_{3r}&= (-1)^r \bigg(q^{\frac{9r^2-5r}{2}}
\frac{\left(\displaystyle{e,q^3u,q^3v,q^4/u,q^4/v;q^3}\right)_r}
   {\left(\displaystyle{q^4/e,u,v,q/u,q/v;q^3}\right)_r\,\displaystyle{e^r}}
+q^{\frac{9r^2-11r}{2}}
\frac{\left(\displaystyle{e/q,q^2u,q^2v,q^3/u,q^3/v;q^3}\right)_r}
   {\left(\displaystyle{q^3/e,u/q,v/q,1/u,1/v;q^3}\right)_r\,\displaystyle{e^r}}\bigg),\\
\alpha_{3r+1}&=(-1)^{r+1}
q^{\frac{9r^2-5r}{2}}\frac{\left(\displaystyle{e,q^3u,q^3v,q^4/u,q^4/v;q^3}\right)_r}
   {\left(\displaystyle{q^4/e,u,v,q/u,q/v;q^3}\right)_r\displaystyle{e^r}},\notag\\
\alpha_{3r-1}&=(-1)^{r+1}q^{\frac{9r^2-11r}{2}}
\frac{\left(\displaystyle{e/q,q^2u,q^2v,q^3/u,q^3/v;q^3}\right)_r}
{\left(\displaystyle{q^3/e,u/q,v/q,1/u,1/v;q^3}\right)_r\displaystyle{e^r}},\notag
\\
\beta_n&= K_2(n) \frac{(q^2/e;q^3)_n
q^n}{(q;q)_{2n}(q^2/e;q)_n},\notag
\end{align}
}
(iii) the pair of sequences $(\alpha_n, \beta_n)$ is a Bailey pair
with respect to $a=q$, where $\alpha_0=\beta_0=1$ and
{\allowdisplaybreaks
\begin{align}\label{chubp55}
\alpha_{3r}&= (-1)^r \frac{1-q^{6r+1}}{1-q}q^{\frac{9r^2-11r}{2}}
\frac{\left(\displaystyle{e,q^3u,q^3v,q^4/u,q^4/v;q^3}\right)_r}
   {\left(\displaystyle{q^4/e,u,v,q/u,q/v;q^3}\right)_r\,\displaystyle{e^r}}&\\
\alpha_{3r+1}&=0,\notag\\
\alpha_{3r-1}&=(-1)^{r+1}\frac{1-q^{6r-1}}{1-q}q^{\frac{9r^2-17r}{2}+1}
\frac{\left(\displaystyle{e/q,q^2u,q^2v,q^3/u,q^3/v;q^3}\right)_r}
{\left(\displaystyle{q^3/e,u/q,v/q,1/u,1/v;q^3}\right)_r\displaystyle{e^r}},\notag
\\
\beta_n&= K_2(n) \frac{(q^2/e;q^3)_n}{(q;q)_{2n}(q^2/e;q)_n},\notag
\end{align}
}

\end{theorem}

\begin{remark}
 The condition $e\not = 1/q$ is necessary to ensure that $\beta_0=1$.
\end{remark}

\begin{proof}
Replace $q$ with $q^3$ in \eqref{chueq1}, and then set $a=q$, $b=q^{-n}$, $c=q^{1-n}$ and $d=q^{2-n}$. One easily checks that the right side simplifies to give
\[
K_2(n)\frac{(q,q^2;q)_n(q^2/e;q^3)_n}{(q;q)_{2n}(q^2/e;q)_n}.
\]

On the series side, the choices for the parameters force the series
to terminate above and below, and we get, after some elementary manipulations, that the series becomes
\[
\sum_{r=-(n+1)/3}^{n/3}\frac{1-q^{6r+1}}{1-q}
\frac{(q^{-n};q)_{3r}}
{(q^{n+2};q)_{3r}}
\frac{\left(\displaystyle{e,q^3u,q^3v,q^4/u,q^4/v;q^3}\right)_r}
{\left(\displaystyle{q^4/e,u,v,q/u,q/v;q^3}\right)_r}
\left(\frac{q^{3n-4}}{e}\right)^r.
\]

Next, we apply \eqref{q-njeq} to the $(q^{-n};q)_{3r}$ factor above and rearrange terms to get
\begin{equation*}
\sum_{r=-\frac{n+1}{3}}^{n/3}
\frac{(1-q^{6r+1})(-1)^rq^{(9r^2-11r)/2}}
{(q;q)_{n+3r+1}(q;q)_{n-3r}\, e^r}\frac{\left(\displaystyle{e,q^3u,q^3v,q^4/u,q^4/v;q^3}\right)_r}
{\left(\displaystyle{q^4/e,u,v,q/u,q/v;q^3}\right)_r}\\
= \frac{K_2(n)(q^2/e;q^3)_n}{(q;q)_{2n}(q^2/e;q)_n}.
\end{equation*}
Noting that, for arbitrary non-zero $y$ and $z$,
\begin{equation}\label{binneg}
\frac{(y;q)_{-r}}{(z;q)_{-r}} = \frac{(q/z;q)_{r}z^{r}}{(q/y;q)_{r}y^{r}},
\end{equation}
we get that
{\allowdisplaybreaks
\begin{multline*}
K_2(n)\frac{(q^2/e;q^3)_n}{(q;q)_{2n}(q^2/e;q)_n}\\
= \sum_{r=0}^{n/3} \frac{(1-q^{6r+1})(-1)^rq^{(9r^2-11r)/2}}
{(q;q)_{n+3r+1}(q;q)_{n-3r} e^r}
\frac{\left(\displaystyle{e,q^3u,q^3v,q^4/u,q^4/v;q^3}\right)_r}
{\left(\displaystyle{q^4/e,u,v,q/u,q/v;q^3}\right)_r}\phantom{asdaDSasdSDasdsaasda}
\\
 +
\sum_{r=1}^{\frac{n+1}{3}}
\frac{(1-q^{-6r+1})(-1)^rq^{(9r^2-5r)/2}}
{(q;q)_{n-3r+1}(q;q)_{n+3r} e^r}
\frac{\left(\displaystyle{e/q,q^2u,q^2v,q^3/u,q^3/v;q^3}\right)_r}
   {\left(\displaystyle{q^3/e,u/q,v/q,1/u,1/v;q^3}\right)_r}\\
=\sum_{r=0}^{n/3}\left(
\frac{q^{(9r^2-11r)/2}}
{(q;q)_{n+3r}(q;q)_{n-3r}} - \frac{q^{(9r^2+r)/2+1}}
{(q;q)_{n+3r+1}(q;q)_{n-3r-1}}\right)
(-1)^r
\frac{\left(\displaystyle{e,q^3u,q^3v,q^4/u,q^4/v;q^3}\right)_r}
{\left(\displaystyle{q^4/e,u,v,q/u,q/v;q^3}\right)_r}
\\
 +
\sum_{r=1}^{\frac{(n+1)}{3}}\left(
\frac{q^{(9r^2-5r)/2}}
{(q;q)_{n-3r}(q;q)_{n+3r}} - \frac{q^{(9r^2-17r)/2+1}}
{(q;q)_{n-3r+1}(q;q)_{n+3r-1}}
\right)
(-1)^r
\frac{\left(\displaystyle{e/q,q^2u,q^2v,q^3/u,q^3/v;q^3}\right)_r}
   {\left(\displaystyle{q^3/e,u/q,v/q,1/u,1/v;q^3}\right)_r},
\end{multline*}
}
where the last equality follows from the identities
\begin{align}\label{6rt1}
(1-q^{6r+1})q^{\frac{9r^2-11r}{2}}&=q^{\frac{9r^2-11r}{2}}(1-q^{n+3r+1})
-q^{\frac{9r^2+r}{2}+1}(1-q^{n-3r}),\\
(1-q^{-6r+1})q^{\frac{9r^2-5r}{2}}&=q^{\frac{9r^2-5r}{2}}(1-q^{n-3r+1})
-q^{\frac{9r^2-17r}{2}+1}(1-q^{n+3r}).\notag
\end{align}
That \eqref{chubp3} gives a Bailey pair now follows from the
definition of a Bailey pair at \eqref{bpeq}, also noting that
$K_2(0)=K_2(1)=1$.

If, instead of using the identities at \eqref{6rt1}, we use the
identities
\begin{align}\label{6rt2}
(1-q^{6r+1})q^{\frac{9r^2-11r}{2}}
&=q^{\frac{9r^2-5r}{2}-n}((1-q^{n+3r+1})
-(1-q^{n-3r})),\\
(1-q^{-6r+1})q^{\frac{9r^2-5r}{2}}
&=q^{\frac{9r^2-11r}{2}-n}((1-q^{n-3r+1}) -(1-q^{n+3r})),\notag
\end{align}
we get that the pair at \eqref{chubp5} is a Bailey pair.

The proof of \eqref{chubp55} follows  from the right side of the
first equality following \eqref{binneg}, upon setting
$(q;q)_{n+3r+1}$ $(q;q)_{n-3r}=(1-q)(q^2;q)_{n+3r}(q;q)_{n-3r}$ in
the first sum and
$(q;q)_{n-3r+1}(q;q)_{n+3r}=(1-q)(q;q)_{n-3r+1}(q^2;q)_{n+3r-1}$ in
the second sum.
\end{proof}

If we begin by setting $a=q^2$ instead of $a=q$, but keep the same
choices $b=q^{-n}$, $c=q^{1-n}$ and $d=q^{2-n}$ in \eqref{chueq1}, then we get
Theorems \ref{t5}  following.

\begin{theorem}\label{t5}
Let  $K_3(n):=K(q^2,q^{-n},q^{1-n},q^{2-n},e,u,v,q^3)$. Then \\
(i) the pair of sequences $(\alpha_n, \beta_n)$ is a Bailey pair
with respect to $a=q$, where $\alpha_0=\beta_0=1$ and
{\allowdisplaybreaks
\begin{align}\label{chubp7}
\alpha_{3r}&=
 (-1)^rq^{\frac{9r^2-7r}{2}}
\frac{\left(\displaystyle{e,q^3u,q^3v,q^5/u,q^5/v;q^3}\right)_r }
   {\left(\displaystyle{q^5/e,u,v,q^2/u,q^2/v;q^3}\right)_r\,\displaystyle{e^r}},\\
\alpha_{3r+1}&= (-1)^{r+1} q^{\frac{9r^2+5r}{2}+2}
\frac{\left(\displaystyle{e,q^3u,q^3v,q^5/u,q^5/v;q^3}\right)_r}
   {\left(\displaystyle{q^5/e,u,v,q^2/u,q^2/v;q^3}\right)_r\,\displaystyle{e^r}}
  +  (-1)^{r}q^{\frac{9r^2-r}{2}-3}
\frac{\left(\displaystyle{e/q^2,qu,qv,q^3/u,q^3/v;q^3}\right)_{r+1}}
{\left(\displaystyle{q^3/e,u/q^2,v/q^2,1/u,1/v;q^3}\right)_{r+1}\displaystyle{e^{r+1}}},\notag\\
\alpha_{3r-1}&= (-1)^{r}q^{\frac{9r^2-7r}{2}}
\frac{\left(\displaystyle{e/q^2,qu,qv,q^3/u,q^3/v;q^3}\right)_{r}}
{\left(\displaystyle{q^3/e,u/q^2,v/q^2,1/u,1/v;q^3}\right)_{r}\displaystyle{e^r}},\notag
\\
\beta_n&= K_3(n)
\frac{(q^4/e;q^3)_n}{(q^2;q)_{2n}(q^3/e;q)_n};\notag
\end{align}
} (ii) the pair of sequences $(\alpha_n, \beta_n)$ is a Bailey pair
with respect to $a=q$, where $\alpha_0=\beta_0=1$ and
{\allowdisplaybreaks
\begin{align}\label{chubp8}
\alpha_{3r}&= (-1)^r q^{\frac{9r^2-r}{2}}
\frac{\left(\displaystyle{e,q^3u,q^3v,q^5/u,q^5/v;q^3}\right)_r}
   {\left(\displaystyle{q^5/e,u,v,q^2/u,q^2/v;q^3}\right)_r\,\displaystyle{e^r}},\\
\alpha_{3r+1}&= (-1)^{r+1} q^{\frac{9r^2-r}{2}}
\frac{\left(\displaystyle{e,q^3u,q^3v,q^5/u,q^5/v;q^3}\right)_r}
   {\left(\displaystyle{q^5/e,u,v,q^2/u,q^2/v;q^3}\right)_r\,\displaystyle{e^r}}
+  (-1)^{r}q^{\frac{9r^2+5r}{2}-2}
\frac{\left(\displaystyle{e/q^2,qu,qv,q^3/u,q^3/v;q^3}\right)_{r+1}}
{\left(\displaystyle{q^3/e,u/q^2,v/q^2,1/u,1/v;q^3}\right)_{r+1}\displaystyle{e^{r+1}}},\notag\\
\alpha_{3r-1}&= (-1)^{r}q^{\frac{9r^2-13r}{2}}
\frac{\left(\displaystyle{e/q^2,qu,qv,q^3/u,q^3/v;q^3}\right)_{r}}
{\left(\displaystyle{q^3/e,u/q^2,v/q^2,1/u,1/v;q^3}\right)_{r}\displaystyle{e^r}},\notag
\\
\beta_n&= K_3(n) \frac{(q^4/e;q^3)_n
q^n}{(q^2;q)_{2n}(q^3/e;q)_n}.\notag
\end{align}
}
(iii) the pair of sequences $(\alpha_n, \beta_n)$ is a Bailey pair
with respect to $a=q^2$, where $\alpha_0=\beta_0=1$, and
{\allowdisplaybreaks
\begin{align}\label{chubp77}
\alpha_{3r}&=
 (-1)^r \frac{1-q^{6r+2}}{1-q^2}q^{\frac{9r^2-7r}{2}}
\frac{\left(\displaystyle{e,q^3u,q^3v,q^5/u,q^5/v;q^3}\right)_r }
   {\left(\displaystyle{q^5/e,u,v,q^2/u,q^2/v;q^3}\right)_r\,\displaystyle{e^r}},\\
\alpha_{3r+1}&=   (-1)^{r}\frac{1-q^{6r+4}}{1-q^2}q^{\frac{9r^2-r}{2}-3}
\frac{\left(\displaystyle{e/q^2,qu,qv,q^3/u,q^3/v;q^3}\right)_{r+1}}
{\left(\displaystyle{q^3/e,u/q^2,v/q^2,1/u,1/v;q^3}\right)_{r+1}\displaystyle{e^{r+1}}},\notag\\
\alpha_{3r-1}&= 0,\notag
\\
\beta_n&= K_3(n)
\frac{(q^4/e;q^3)_n}{(q^2;q)_{2n}(q^3/e;q)_n};\notag
\end{align}
}
\end{theorem}

\begin{proof}
The proof parallels that of Theorem \ref{t3}, except that after arriving at the identity
\begin{equation}\label{t5bl}
\sum
\frac{(1-q^{6r+2})(-1)^rq^{(9r^2-7r)/2}}
{(q^2;q)_{n+3r+1}(q;q)_{n-3r}\, e^r}\frac{\left(\displaystyle{e,q^3u,q^3v,q^5/u,q^5/v;q^3}\right)_r}
{\left(\displaystyle{q^5/e,u,v,q^2/u,q^2/v;q^3}\right)_r}\\
= \frac{K_3(n)(q^4/e;q^3)_n}{(q^2;q)_{2n}(q^3/e;q)_n},
\end{equation}
and then separating the sum into two sums ($r\geq 0$ and $r<0$) as
previously,  we instead employ the identities
\begin{align}\label{6rt3}
(1-q^{6r+2})q^{\frac{9r^2-7r}{2}}&
=q^{\frac{9r^2-7r}{2}}(1-q^{n+3r+2})
-q^{\frac{9r^2+5r}{2}+2}(1-q^{n-3r}),\\
(1-q^{-6r+2})q^{\frac{9r^2-7r}{2}}&
=q^{\frac{9r^2-7r}{2}}(1-q^{n-3r+2})
-q^{\frac{9r^2-19r}{2}+2}(1-q^{n+3r}), \notag
\end{align}
to get \eqref{chubp7}.  The result  follows as in Theorem
\ref{t3}, except that it is necessary to re-index one of the four
sums (by replacing $r$ with $r+1$).

For \eqref{chubp8} we instead use the identities
\begin{align}\label{6rt4}
(1-q^{6r+2})q^{\frac{9r^2-7r}{2}}&
=q^{\frac{9r^2-r}{2}-n}((1-q^{n+3r+2})
-(1-q^{n-3r})),\\
(1-q^{-6r+2})q^{\frac{9r^2-7r}{2}}&
=q^{\frac{9r^2-13r}{2}-n}((1-q^{n-3r+2}) -(1-q^{n+3r})).\notag
\end{align}

The proof of \eqref{chubp77} is like the proof of \eqref{chubp55} in Theorem \ref{t3},
except that after separating the sum at \eqref{t5bl} into two sums, according to $r\geq 0$ or $r<0$, and then replacing $r$ with $-r$ for the sum with $r<0$, we use the identities
$(q^2;q)_{n+3r+1}(q;q)_{n-3r}=(1-q^2)(q^3;q)_{n+3r}(q;q)_{n-3r}$  and
$(q^2;q)_{n-3r+1}(q;q)_{n-3r}=(1-q^2)(q;q)_{n-3r+2}(q^3;q)_{n+3r-2}$, and finally re-index in the latter case by replacing
$r$ with $r+1$.
\end{proof}

The nine Bailey pairs in the next corollary derive, respectively,
from the pairs in Theorems \ref{t23} - \ref{t5}, by letting $u$, $v
\to \infty$ in each case.
\begin{corollary}\label{csgen1}
The   sequences $(\alpha_n, \beta_n)$ below are  Bailey pairs with
respect to the stated values of $a$, where
{\allowdisplaybreaks
\begin{align}\label{sgen20}
\alpha_{3r}&= \frac{1-a q^{6r}}{1-a}\frac{(a;q^3)_r}{(q^3;q^3)_r}q^{\frac{9r^2-3r}{2}}(-a)^r \\
\alpha_{3r\pm1}&=0,\notag\\
\beta_n&= \frac{(a;q^3)_{n}}{(a;q)_{2n}(q;q)_n}, \,\, \text{ with
respect to $a=a$};\notag
\end{align}
}
{\allowdisplaybreaks
\begin{align}\label{sgen21}
\alpha_{3r}&= \frac{\left(\displaystyle{a q;q^3}\right)_r}
   {\left(\displaystyle{q^3;q^3}\right)_r}q^{\frac{9r^2-r}{2}}(-a)^r,\\
\alpha_{3r+1}&= \frac{\left(\displaystyle{a q;q^3}\right)_r}
   {\left(\displaystyle{q^3;q^3}\right)_r}q^{\frac{9r^2+11r}{2}+1}(-a)^{r+1},\notag\\
   \alpha_{3r-1}&=0, \notag\\
   \beta_n&= \frac{(aq ;q^3)_{n}}{(aq;q)_{2n}(q;q)_n}, \,\, \text{
with respect to $a=a$;
   }\notag
\end{align}
} {\allowdisplaybreaks
\begin{align}\label{sgen22}
\alpha_{3r}&= \frac{\left(\displaystyle{a q;q^3}\right)_r}
   {\left(\displaystyle{q^3;q^3}\right)_r}q^{\frac{9r^2+5r}{2}}(-a)^r,\\
\alpha_{3r+1}&= -\frac{\left(\displaystyle{a q;q^3}\right)_r}
   {\left(\displaystyle{q^3;q^3}\right)_r}q^{\frac{9r^2+5r}{2}}(-a)^{r},\notag\\
   \alpha_{3r-1}&=0, \notag\\
   \beta_n&=  \frac{(aq ;q^3)_{n}q^n}{(aq;q)_{2n}(q;q)_n}, \,\, \text{
with respect to $a=a$;
   }\notag
\end{align}
}
{\allowdisplaybreaks
\begin{align}\label{Sgen2}
\alpha_{3r}&=(-1)^r\left(\displaystyle{
\frac{q^{9r^2/2+r/2}(e;q^3)_r}{(q^4/e;q^3)_r e^r}+
\frac{q^{9r^2/2+7r/2}(e/q;q^3)_r}{(q^3/e;q^3)_r e^r}}\right),\\
\alpha_{3r+1}&=\displaystyle{\frac{(-1)^{r+1}q^{9r^2/2+13r/2+1}(e;q^3)_r}{(q^4/e;q^3)_r e^r}},\notag\\
\alpha_{3r-1}&=\displaystyle{\frac{(-1)^{r+1}q^{9r^2/2-5r/2+1}(e/q;q^3)_r}{(q^3/e;q^3)_r
e^r}},\notag\\
\beta_n&=\frac{(q^2/e;q^3)_n}{(q;q)_{2n}(q^2/e;q)_n}, \,\, \text{
with respect to $a=1$;
   }\notag
\end{align}
} {\allowdisplaybreaks
\begin{align}\label{Sgen3}
\alpha_{3r}&=(-1)^r\left(\displaystyle{
\frac{q^{9r^2/2+7r/2}(e;q^3)_r}{(q^4/e;q^3)_r e^r}+
\frac{q^{9r^2/2+r/2}(e/q;q^3)_r}{(q^3/e;q^3)_r e^r}}\right),\\
\alpha_{3r+1}&=\displaystyle{\frac{(-1)^{r+1}q^{9r^2/2+7r/2}(e;q^3)_r}{(q^4/e;q^3)_r e^r}},\notag\\
\alpha_{3r-1}&=\displaystyle{\frac{(-1)^{r+1}q^{9r^2/2+r/2}(e/q;q^3)_r}{(q^3/e;q^3)_r
e^r}},\notag\\
\beta_n&=\frac{q^n(q^2/e;q^3)_n}{(q;q)_{2n}(q^2/e;q)_n},\,\, \text{
with respect to $a=1$;
 }\notag
\end{align}
}
{\allowdisplaybreaks
\begin{align}\label{sgen35}
\alpha_{3r}&= (-1)^r \frac{1-q^{6r+1}}{1-q}q^{\frac{9r^2+r}{2}}
\frac{\left(\displaystyle{e;q^3}\right)_r}
   {\left(\displaystyle{q^4/e;q^3}\right)_r\,\displaystyle{e^r}}&\\
\alpha_{3r+1}&=0,\notag\\
\alpha_{3r-1}&=(-1)^{r+1}\frac{1-q^{6r-1}}{1-q}q^{\frac{9r^2-5r}{2}+1}
\frac{\left(\displaystyle{e/q;q^3}\right)_r}
{\left(\displaystyle{q^3/e;q^3}\right)_r\displaystyle{e^r}},\notag
\\
\beta_n&=  \frac{(q^2/e;q^3)_n}{(q;q)_{2n}(q^2/e;q)_n},\,\, \text{
with respect to $a=q$;
 }\notag
\end{align}
}
{\allowdisplaybreaks
\begin{align}\label{Sgen4}
\alpha_{3r}&= \displaystyle{
\frac{(-1)^r q^{9r^2/2+5r/2}(e;q^3)_r}{(q^5/e;q^3)_r e^r}},\\
\alpha_{3r+1}&=\displaystyle{\frac{(-1)^{r}q^{9r^2/2+11r/2+3}(e/q^2;q^3)_{r+1}}
{(q^3/e;q^3)_{r+1} e^{r+1}}}
+\displaystyle{\frac{(-1)^{r+1}q^{9r^2/2+17r/2+2}(e;q^3)_{r}}
{(q^5/e;q^3)_{r} e^{r}}},\notag\\
\alpha_{3r-1}&=\displaystyle{\frac{(-1)^{r}q^{9r^2/2+5r/2}(e/q^2;q^3)_r}{(q^3/e;q^3)_r
e^r}},\notag\\
\beta_n&=\frac{(q^4/e;q^3)_n}{(q^2;q)_{2n}(q^3/e;q)_n}, \,\,\text{
 with respect to $a=q$; }\notag
\end{align}
} {\allowdisplaybreaks
\begin{align}\label{Sgen5}
\alpha_{3r}&= \displaystyle{
\frac{(-1)^r q^{9r^2/2+11r/2}(e;q^3)_r}{(q^5/e;q^3)_r e^r}},\\
\alpha_{3r+1}&=\displaystyle{\frac{(-1)^{r}q^{9r^2/2+17r/2+4}(e/q^2;q^3)_{r+1}}
{(q^3/e;q^3)_{r+1} e^{r+1}}}
+\displaystyle{\frac{(-1)^{r+1}q^{9r^2/2+11r/2}(e;q^3)_{r}}
{(q^5/e;q^3)_{r} e^{r}}},\notag\\
\alpha_{3r-1}&=\displaystyle{\frac{(-1)^{r}q^{9r^2/2-r/2}(e/q^2;q^3)_r}{(q^3/e;q^3)_r
e^r}},\notag\\
\beta_n&=\frac{q^n(q^4/e;q^3)_n}{(q^2;q)_{2n}(q^3/e;q)_n},\,\,
\text{  with respect to $a=q$; } \notag
\end{align}
}
{\allowdisplaybreaks
\begin{align}\label{sgen37}
\alpha_{3r}&=
 (-1)^r \frac{1-q^{6r+2}}{1-q^2}q^{\frac{9r^2+5r}{2}}
\frac{\left(\displaystyle{e;q^3}\right)_r }
   {\left(\displaystyle{q^5/e;q^3}\right)_r\,\displaystyle{e^r}},\\
\alpha_{3r+1}&=   (-1)^{r}\frac{1-q^{6r+4}}{1-q^2}q^{\frac{9r^2+11r}{2}+3}
\frac{\left(\displaystyle{e/q^2;q^3}\right)_{r+1}}
{\left(\displaystyle{q^3/e;q^3}\right)_{r+1}\displaystyle{e^{r+1}}},\notag\\
\alpha_{3r-1}&= 0,\notag
\\
\beta_n&=
\frac{(q^4/e;q^3)_n}{(q^2;q)_{2n}(q^3/e;q)_n}\,\,
\text{  with respect to $a=q^2$. }\notag
\end{align}
}
\end{corollary}

All  eight pairs in Slater's \textbf{A} table (\cite[page 463]{S51})
and all six on her \textbf{J} list (\cite[pp. 148--149]{S52})
are  derived from the five Bailey pairs \eqref{sgen20}, \eqref{Sgen2}, \eqref{Sgen3}, \eqref{Sgen4}, \eqref{Sgen5} above, for particular
values of $e$.  Slater  did not write down these  general pairs
above explicitly, but she could easily have derived them by the same
methods she used to derive the special cases. None of Slater's pairs arise as special cases of \eqref{sgen21}, \eqref{sgen22}, \eqref{sgen35} or \eqref{sgen37},  although the special case
$e=-q^2$ of \eqref{sgen35} was given in \cite{McLS08}. However, specializing the parameters give Bailey pairs  which then give rise to some of the series-product identities on Slater's list, showing that different Bailey pairs may lead the same identity of Rogers-Ramanujan type.

Each of the nine Bailey pairs above also gives rise to a
transformation between basic hypergeometric series, upon
substituting the pair into \eqref{Seq1}. The pair at \eqref{sgen20}, for example, gives the following identity.

\begin{corollary}
{\allowdisplaybreaks
\begin{equation}\label{c6a}
\sum_{n=0}^{\infty} \frac{(y,z;q)_n(a;q^3)_{n}}
{(a;q)_{2n}(q;q)_{n}}\left ( \frac{aq}{yz}\right)^n
=\frac{(aq/y,aq/z;q)_{\infty}}{(a q,aq/yz;q)_{\infty}} \\
\sum_{n=0}^{\infty}\frac{(1-a q^{6n})(y,z;q)_{3n}(a;q^3)_n a^{4n}q^{9n^2/2+3n/2 }}
{(1-a )(aq/y,aq/z;q)_{3n}(q^3;q^3)_n}\left(\frac{-1}{y z}
\right)^{3n}.
\end{equation}
}
\end{corollary}

\subsection{Mod 2 Bailey Pairs}

We next consider how Slater produced the Bailey pairs in the \textbf{G-}, \textbf{C-} and \textbf{I} tables of \cite{S51, S52}.

\begin{theorem}\label{t66}
Let $K_4(n):=K(a,q^{-n},q^{1-n},d,a,u,v,q^2)$. Then  \\
(i) the pair of sequences $(\alpha_n, \beta_n)$ is a Bailey pair
with respect to $a=a$, where $\alpha_0=\beta_0=1$ and
{\allowdisplaybreaks
\begin{align}\label{chubp66}
\alpha_{2r}&= \frac{\displaystyle{1-a q^{4r}}}{\displaystyle{1-a}}
\frac{\left(\displaystyle{a,d,q^2u,q^2v,a q^2/u,a
q^2/v;q^2}\right)_r \displaystyle{q^{2r^2-4r}}a^r}
   {\left(\displaystyle{a q^2/d,u,v,a/u,a/v,q^2;q^2}\right)_r\displaystyle{d^r}},\\
\alpha_{2r-1}&=0,\notag
\\
\beta_n&= K_4(n) \frac{(a q/d;q^2)_n}{(a q;q^2)_{n}(aq/d,q;q)_n}.\notag
\end{align}
}
Let $K_4'(n):=K(a q,q^{-n},q^{1-n},d,a q,u,v,q^2)$. Then  \\
(ii) the pair of sequences $(\alpha_n, \beta_n)$ is a Bailey pair
with respect to $a=a$, where $\alpha_0=\beta_0=1$ and
\begin{align}\label{chubp66b}
\alpha_{2r}&= \frac{\left(\displaystyle{a q,d,q^2u,q^2v,a q^3/u,a
q^3/v;q^2}\right)_r \displaystyle{q^{2r^2-3r}}a^r}
   {\left(\displaystyle{a q^3/d,u,v,a q/u,a q/v,q^2;q^2}\right)_r\displaystyle{d^r}},\\
\alpha_{2r+1}&=- \frac{\left(\displaystyle{a q,d,q^2u,q^2v,a
q^3/u,a q^3/v;q^2}\right)_r \displaystyle{q^{2r^2+r+1}}a^{r+1}}
   {\left(\displaystyle{a q^3/d,u,v,a q/u,a q/v,q^2;q^2}\right)_r\displaystyle{d^r}},\notag
\\
\beta_n&= K_4(n)' \frac{(a q^2/d;q^2)_n}{(a
q^2;q^2)_{n}(aq^2/d,q;q)_n};\notag
\end{align}
(iii) the pair of sequences $(\alpha_n, \beta_n)$ is a Bailey pair
with respect to $a=a$, where $\alpha_0=\beta_0=1$ and
\begin{align}\label{chubp66c}
\alpha_{2r}&= \frac{\left(\displaystyle{a q,d,q^2u,q^2v,a q^3/u,a
q^3/v;q^2}\right)_r \displaystyle{q^{2r^2-r}}a^r}
   {\left(\displaystyle{a q^3/d,u,v,a q/u,a q/v,q^2;q^2}\right)_r\displaystyle{d^r}},\\
\alpha_{2r+1}&=-\frac{\left(\displaystyle{a q,d,q^2u,q^2v,a q^3/u,a
q^3/v;q^2}\right)_r \displaystyle{q^{2r^2-r}}a^r}
   {\left(\displaystyle{a q^3/d,u,v,a q/u,a q/v,q^2;q^2}\right)_r\displaystyle{d^r}},\notag
\\
\beta_n&= K_4(n)' \frac{(a q^2/d;q^2)_n q^n}{(a
q^2;q^2)_{n}(aq^2/d,q;q)_n};\notag
\end{align}
\end{theorem}

\begin{proof}
The proof is quite similar to the proof of Theorem \ref{t23}.
As in the proof of that theorem, set $e=a$ in \eqref{chueq1}, so that all the terms of negative
index vanish. Then replace $q$ with $q^2$, set $b=q^{-n}$ and
$d=q^{1-n}$. After some simple
manipulations,  \eqref{chueq1} becomes
\begin{equation*}
\sum_{r=0}^{n/2}\frac{1-a q^{4r}}{1-a}
\frac{(q^{-n};q)_{2r}}{(a q^{n+1};q)_{2r}}
\frac{\left(\displaystyle{a,d,q^2u,q^2v,aq^2/u,aq^2/v;q^2}\right)_r}
{\left(\displaystyle{a q^2/du,v,a/u,a/v,q^2;q^2}\right)_r}\left(\frac{a q^{2n-3}}{d}
\right)^r
= K_4(n) \frac{(aq/d;q^2)_{n}(a q;q)_n}{(a q;q^2)_{n}(aq/d;q)_n}.
\end{equation*}
Apply \eqref{q-njeq} to the $(q^{-n};q)_{2r}$ factor, divide both
sides by $(aq;q)_n(q;q)_n$ to get
\begin{equation}\label{pr66eq}
\sum_{r=0}^{n/2}\frac{1-a q^{4r}}{1-a} \frac{q^{2r^2-4r}}{(a
q;q)_{n+2r}(q;q)_{n-2r}}
\frac{\left(\displaystyle{a,d,q^2u,q^2v,aq^2/u,aq^2/v;q^2}\right)_r}
{\left(\displaystyle{a
q^2/du,v,a/u,a/v,q^2;q^2}\right)_r}\left(\frac{a}{d}
\right)^r
= K_4(n) \frac{(aq/d;q^2)_{n}}{(a q;q^2)_{n}(aq/d,q;q)_n},
\end{equation}
and the proof for the pair at \eqref{chubp66} follows.

For \eqref{chubp66b} and \eqref{chubp66c}, replace $a$ with $aq$ in
\eqref{pr66eq}, and use, respectively, the identities
\begin{align*}
(1-aq^{4r+1})q^{2r^2-3r}&=q^{2r^2-3r}(1-a q^{n+2r+1})-aq^{2r^2+r+1}(1-q^{n-2r}),\\
(1-aq^{4r+1})q^{2r^2-3r}&=q^{2r^2-r}q^{-n}((1-aq^{n+2r+1})-(1-q^{n-2r})).
\end{align*}

\end{proof}

\begin{theorem}\label{t7}
Let $K_5(n):=K(q,q^{-n},q^{1-n},d,e,u,v,q^2)$. Then \\
(i) the pair of sequences $(\alpha_n, \beta_n)$ is a Bailey pair with
respect to $a=q$, where $\alpha_0=\beta_0=1$ and
\begin{align}\label{chubp9}
\alpha_{2r}&=
\frac{\displaystyle{1-q^{4r+1}}}{\displaystyle{1-q}}
\frac{\left(\displaystyle{d,e,q^2u,q^2v,q^3/u,q^3/v;q^2}\right)_r\displaystyle{q^{2r^2-2r}}}
   {\left(\displaystyle{q^3/d,q^3/e,u,v,q/u,q/v;q^2}\right)_r\displaystyle{(d \,e)^r}},\\
\alpha_{2r-1}&=-\frac{\displaystyle{1-q^{4r-1}}}{\displaystyle{1-q}}
\frac{\left(\displaystyle{d/q,e/q,qu,qv,q^2/u,q^2/v;q^2}\right)_r\displaystyle{q^{2r^2-4r+1}}}
   {\left(\displaystyle{q^2/d,q^2/e,u/q,v/q,1/u,1/v;q^2}\right)_r\displaystyle{(de)^r}},\notag
\\
\beta_n&= K_5(n) \frac{(q^3/de;q^2)_n}{(q^2;q^2)_{n}(q^2/d,q^2/e;q)_n};\notag
\end{align}
(ii) The pair of sequences $(\alpha_n, \beta_n)$ is a Bailey pair
with respect to $a=1$, where $\alpha_0=\beta_0=1$ and
{\allowdisplaybreaks
\begin{align}\label{chubp10}
\alpha_{2r}&=\frac{\left(\displaystyle{d,e,q^2u,q^2v,q^3/u,q^3/v;q^2}\right)_r\displaystyle{q^{2r^2-2r}}}
   {\left(\displaystyle{q^3/d,q^3/e,u,v,q/u,q/v;q^2}\right)_r\displaystyle{(d \,e)^r}}
+\frac{\left(\displaystyle{d/q,e/q,qu,qv,q^2/u,q^2/v;q^2}\right)_r\displaystyle{q^{2r^2}}}
   {\left(\displaystyle{q^2/d,q^2/e,u/q,v/q,1/u,1/v;q^2}\right)_r\displaystyle{(de)^r}},\\
\alpha_{2r-1}&=-\frac{\left(\displaystyle{d,e,q^2u,q^2v,q^3/u,q^3/v;q^2}\right)_{r-1}\displaystyle{q^{2r^2-2r+1}}}
   {\left(\displaystyle{q^3/d,q^3/e,u,v,q/u,q/v;q^2}\right)_{r-1}\displaystyle{(d \,e)^{r-1}}}
-
\frac{\left(\displaystyle{d/q,e/q,qu,qv,q^2/u,q^2/v;q^2}\right)_r\displaystyle{q^{2r^2-4r+1}}}
   {\left(\displaystyle{q^2/d,q^2/e,u/q,v/q,1/u,1/v;q^2}\right)_r\displaystyle{(de)^r}},\notag\\
\beta_n&= K_5(n) \frac{(q^3/de;q^2)_n}{(q^2;q^2)_{n}(q^2/d,q^2/e;q)_n};\notag
\end{align}
} (iii) The pair of sequences $(\alpha_n, \beta_n)$ is a Bailey
pair with respect to $a=1$, where $\alpha_0=\beta_0=1$ and
{\allowdisplaybreaks
\begin{align}\label{chubp11}
\alpha_{2r}&=\frac{\left(\displaystyle{d,e,q^2u,q^2v,q^3/u,q^3/v;q^2}\right)_r\displaystyle{q^{2r^2}}}
   {\left(\displaystyle{q^3/d,q^3/e,u,v,q/u,q/v;q^2}\right)_r\displaystyle{(d \,e)^r}}
+\frac{\left(\displaystyle{d/q,e/q,qu,qv,q^2/u,q^2/v;q^2}\right)_r\displaystyle{q^{2r^2-2r}}}
   {\left(\displaystyle{q^2/d,q^2/e,u/q,v/q,1/u,1/v;q^2}\right)_r\displaystyle{(de)^r}},
  \\
\alpha_{2r-1}&=-\frac{\left(\displaystyle{d,e,q^2u,q^2v,q^3/u,q^3/v;q^2}\right)_{r-1}\displaystyle{q^{2r^2-4r+2}}}
   {\left(\displaystyle{q^3/d,q^3/e,u,v,q/u,q/v;q^2}\right)_{r-1}\displaystyle{(d \,e)^{r-1}}}
-
\frac{\left(\displaystyle{d/q,e/q,qu,qv,q^2/u,q^2/v;q^2}\right)_r\displaystyle{q^{2r^2-2r}}}
   {\left(\displaystyle{q^2/d,q^2/e,u/q,v/q,1/u,1/v;q^2}\right)_r\displaystyle{(de)^r}},\notag\\
\beta_n&= K_5(n)\,\frac{ q^n \,(q^3/de;q^2)_n}{(q^2;q^2)_{n}(q^2/d,q^2/e;q)_n};\notag
\end{align}
}
\end{theorem}

\begin{proof}
The proof is very similar to the proof of Theorem \ref{t3}.
This time, replace $q$ with $q^2$ in \eqref{chueq1}, and then set $a=q$, $b=q^{-n}$, $c=q^{1-n}$. The right side simplifies to give
\[
K_5(n)\frac{(q,q^2;q)_n(q^3/de;q^2)_n}{(q^2;q^2)_{n}(q^2/d,q^2/e;q)_n}.
\]

After some elementary manipulations,  the series becomes
\[
\sum_{r=-(n+1)/2}^{n/2}\frac{1-q^{4r+1}}{1-q}
\frac{(q^{-n};q)_{2r}}
{(q^{n+2};q)_{2r}}
\frac{\left(\displaystyle{d,e,q^2u,q^2v,q^3/u,q^3/v;q^2}\right)_r}
{\left(\displaystyle{q^3/d,q^3/e,u,v,q/u,q/v;q^2}\right)_r}
\left(\frac{q^{2n-1}}{d e}\right)^r.
\]

We apply \eqref{q-njeq} to the $(q^{-n};q)_{2r}$ factor above and rearrange terms to get
\[
\sum_{r=-\frac{n+1}{2}}^{n/2}
\frac{(1-q^{4r+1})q^{2r^2-2r}}
{(1-q)(q^2;q)_{n+2r}(q;q)_{n-2r}\, d^r e^r}
\frac{\left(\displaystyle{d,e,q^2u,q^2v,q^3/u,q^3/v;q^2}\right)_r}
{\left(\displaystyle{q^3/d,q^3/e,u,v,q/u,q/v;q^2}\right)_r}
= \frac{K_5(n)(q^3/de;q^3)_n}{(q^2;q^2)_{n}(q^2/d,q^2/e;q)_n}.
\]
After applying \eqref{binneg} to the terms of negative index in the sum above, we get that
{\allowdisplaybreaks
\begin{multline}\label{pr910eq}
K_5(n)\frac{(q^3/de;q^2)_n}{(q^2;q^2)_{n}(q^2/d,q^2/e;q)_n}\\
=\sum_{r=0}^{n/2} \frac{(1-q^{4r+1})q^{2r^2-2r}}
{(1-q)(q^2;q)_{n+2r}(q;q)_{n-2r}\, d^r e^r}
\frac{\left(\displaystyle{d,e,q^2u,q^2v,q^3/u,q^3/v;q^2}\right)_r}
{\left(\displaystyle{q^3/d,q^3/e,u,v,q/u,q/v;q^2}\right)_r}\\
-
\sum_{r=1}^{\frac{n+1}{2}}
\frac{(1-q^{4r-1})q^{2r^2-4r+1}}
{(1-q)(q^2;q)_{n-2r}(q;q)_{n+2r}\, d^r e^r}
\frac{\left(\displaystyle{d/q,e/q,q u,q v,q^2/u,q^2/v;q^2}\right)_r}
{\left(\displaystyle{q^2/d,q^2/e,u/q,v/q,1/u,1/v;q^2}\right)_r}.
\end{multline}
}
The pair at \eqref{chubp9} now follows after some simple manipulations, noting that
\[(
q^2;q)_{n-2r}(q;q)_{n+2r}=(q;q)_{n-2r+1}(q^2;q)_{n+2r-1}.
\]

The pair at \eqref{chubp10} follows from \eqref{pr910eq}, upon absorbing the $1-q$ factor in the denominators there, employing the identities
\begin{align}\label{t10eqs}
(1-q^{4r+1})q^{2r^2-2r}&=q^{2r^2-2r}(1-q^{n+2r+1})-q^{2r^2+2r+1}(1-q^{n-2r}),\\
(1-q^{4r-1})q^{2r^2-4r+1}&=q^{2r^2-4r+1}(1-q^{n+2r})-q^{2r^2}(1-q^{n-2r+1}),\notag
\end{align}
in the way similar to the way that the pair of identities at \eqref{6rt1} was used in the proof of Theorem \ref{t3}, and finally re-indexing one of the resulting sums (by replacing $r$ with $r-1$).

The proof for the pair at \eqref{chubp11} is similar to the proof for the pair at \eqref{chubp10}, except we employ the identities
\begin{align}\label{t10eqs2}
(1-q^{4r+1})q^{2r^2-2r}&=q^{2r^2-n}((1-q^{n+2r+1})-(1-q^{n-2r})),\\
(1-q^{4r-1})q^{2r^2-4r+1}&=q^{2r^2-2r-n}((1-q^{n+2r})-(1-q^{n-2r+1})).\notag
\end{align}
\end{proof}

As with $K_0$ - $K_3$, $K_4$ and $K_5$ are in general quite complicated, but simplify considerably for particular values of the parameters, leading (as was the case in Corollary \ref{C65}) to  transformations of basic hypergeometric series. We give one example.

\begin{corollary} If $a$, $y$, $z$, $q\in \mathbb{C}$ such that $|q|<1$ and none of the denominators below vanish, then
\begin{equation}
\sum_{n=0}^{\infty}
\frac{(y,z;q)_n (-1/q;q^2)_{n}}{(q^2;q^2)_n (a q;q^2)_n}
\left( \frac{a q^2}{y z}\right)^n=
\frac{(aq/y,aq/z;q)_{\infty}}{(a q,aq/yz;q)_{\infty}}
\sum_{n=0}^{\infty}\frac{1-a q^{4n}}{1-a} \frac{(y,z;q)_{2n}(a^2,-a q^4;q^4)_n}
{(aq/y,aq/z;q)_{2n}(-a,q^4;q^4)_n}\left ( \frac{-a^2}{y^2 z^2} \right)^n q^{2 n^2}.
\end{equation}
\end{corollary}

\begin{proof}
In the Bailey pair at \eqref{chubp66}, set $d=-a/q^2$, $u=-q^2$ and $v=i \sqrt{a}$, so that $K_4(n)$ takes the value
\[
K\left(a,q^{-n},q^{1-n},\frac{-a}{q^2},a,-q^2,i \sqrt{a},q^2\right)=\frac{q^{n-1} (q+1) \left(q^{n+1}+1\right) \left(q^{n+2}+1\right)}{\left(q^2+1\right) \left(q^{2 n-1}+1\right) \left(q^{2 n+1}+1\right)}.
\]
Substitute the resulting Bailey pair into \eqref{Seq1}, and the result follows after some elementary $q$-product manipulations.
\end{proof}

The Bailey pairs in Slater's \textbf{G-}, \textbf{C-} and \textbf{I}
tables, as well as pairs \textbf{E(1)}, \textbf{E(2)} ,
\textbf{E(4)} and \textbf{E(5)} (see \cite[pages 469 and 470]{S51}),
are derived from the next six Bailey pairs. These, in turn, are derived from the pairs in Theorems \ref{t66} and \ref{t7}, by letting $u, v \to \infty$.
\begin{corollary}\label{csgen2}
The   sequences $(\alpha_n, \beta_n)$ below are  Bailey pairs with
respect to the stated values of $a$, where $\alpha_0=\beta_0=1$ and
{\allowdisplaybreaks
\begin{align}\label{sgen660}
\alpha_{2r}&=
\frac{\displaystyle{1-a q^{4r}}}{\displaystyle{1-a}}
\frac{\left(\displaystyle{a,d;q^2}\right)_r a^r\displaystyle{q^{2r^2}}}
   {\left(\displaystyle{a q^2/d,q^2;q^2}\right)_r\displaystyle{d^r}},\\
\alpha_{2r-1}&=0,\notag
\\
\beta_n&= \frac{(a q/d;q^2)_n}{(a q;q^2)_{n}(aq/d,q;q)_n},\,\, \text{ with respect to $a=a$;}\notag
\end{align}
}
{\allowdisplaybreaks
\begin{align}\label{sgen661}
\alpha_{2r}&= \frac{\left(\displaystyle{a q,d;q^2}\right)_r \displaystyle{q^{2r^2+r}}a^r}
   {\left(\displaystyle{a q^3/d,q^2;q^2}\right)_r\displaystyle{d^r}},\\
\alpha_{2r+1}&=- \frac{\left(\displaystyle{a q,d;q^2}\right)_r \displaystyle{q^{2r^2+5r+1}}a^{r+1}}
   {\left(\displaystyle{a q^3/d,q^2;q^2}\right)_r\displaystyle{d^r}},\notag
\\
\beta_n&=\frac{(a q^2/d;q^2)_n}{(a
q^2;q^2)_{n}(aq^2/d,q;q)_n}\,\, \text{ with respect to $a=a$;}\notag
\end{align}
}{\allowdisplaybreaks
\begin{align}\label{sgen662}
\alpha_{2r}&= \frac{\left(\displaystyle{a q,d;q^2}\right)_r \displaystyle{q^{2r^2+3r}}a^r}
   {\left(\displaystyle{a q^3/d,q^2;q^2}\right)_r\displaystyle{d^r}},\\
\alpha_{2r+1}&=-\frac{\left(\displaystyle{a q,d;q^2}\right)_r \displaystyle{q^{2r^2+3r}}a^r}
   {\left(\displaystyle{a q^3/d,q^2;q^2}\right)_r\displaystyle{d^r}},\notag
\\
\beta_n&=\frac{(a q^2/d;q^2)_n q^n}{(a
q^2;q^2)_{n}(aq^2/d,q;q)_n},\,\, \text{ with respect to $a=a$;}\notag
\end{align}
}
{\allowdisplaybreaks
\begin{align}\label{Sgen7}
\alpha_{2r}&=
\displaystyle{ \frac{(1-q^{4r+1})
q^{2r^2+2r}(d;q^2)_r(e;q^2)_r}
{(1-q)(q^3/d;q^2)_r (q^3/e;q^2)_r d^r e^r}},\\
\alpha_{2r-1}&=\displaystyle{-\frac{(1-q^{4r-1}) q^{2r^2+1}(d/q;q^2)_r(e/q;q^2)_r}
{(1-q)(q^2/d;q^2)_r (q^2/e;q^2)_r d^r e^r}},\notag
\\
\beta_n&=\frac{(q^3/de;q^2)_n}{(q^2;q^2)_{n}(q^2/d;q)_n(q^2/e;q)_n},\,\, \text{ with respect to $a=q$;}\notag
\end{align}
} {\allowdisplaybreaks
\begin{align}\label{Sgen8}
\alpha_{2r}&=
\displaystyle{ \frac{
q^{2r^2+4r}(d/q;q^2)_r(e/q;q^2)_r} {(q^2/d;q^2)_r (q^2/e;q^2)_r d^r
e^r}}
+\displaystyle{\frac{
q^{2r^2+2r}(d;q^2)_r(e;q^2)_r}
{(q^3/d;q^2)_r (q^3/e;q^2)_r d^r e^r}},\\
\alpha_{2r+1}&=\displaystyle{ -\frac{ q^{2r^2+4r+3}(d/q;q^2)_{r+1}(e/q;q^2)_{r+1}}
{(q^2/d;q^2)_{r+1} (q^2/e;q^2)_{r+1} d^{r+1} e^{r+1}}}
-\displaystyle{\frac{
q^{2r^2+6r+1}(d;q^2)_r(e;q^2)_r} {(q^3/d;q^2)_r (q^3/e;q^2)_r d^r
e^r}},\notag
\\
\beta_n&=\frac{(q^3/de;q^2)_n}{(q^2;q^2)_{n}(q^2/d;q)_n(q^2/e;q)_n},\,\, \text{ with respect to $a=1$;}\notag
\end{align}
} {\allowdisplaybreaks
\begin{align}\label{Sgen9}
\alpha_{2r}&= \displaystyle{ \frac{
q^{2r^2+2r}(d/q;q^2)_r(e/q;q^2)_r} {(q^2/d;q^2)_r (q^2/e;q^2)_r d^r
e^r}}
 +\displaystyle{\frac{
q^{2r^2+4r}(d;q^2)_r(e;q^2)_r}
{(q^3/d;q^2)_r (q^3/e;q^2)_r d^r e^r}},\\
\alpha_{2r+1}&= \displaystyle{- \frac{ q^{2r^2+6r+4}(d/q;q^2)_{r+1}(e/q;q^2)_{r+1}}
{(q^2/d;q^2)_{r+1} (q^2/e;q^2)_{r+1} d^{r+1} e^{r+1}}}
-\displaystyle{\frac{
q^{2r^2+4r}(d;q^2)_r(e;q^2)_r} {(q^3/d;q^2)_r (q^3/e;q^2)_r d^r
e^r}},\notag\\
\beta_n&=\frac{q^n(q^3/de;q^2)_n}{(q^2;q^2)_{n}(q^2/d;q)_n(q^2/e;q)_n},\,\, \text{ with respect to $a=1$}.\notag
\end{align}
}
\end{corollary}

\subsection{Mod 4 Bailey Pairs}
We continue to follow in Slater's path \cite{S51}, next considering how she produced the Bailey pairs in her \textbf{K} table.
\begin{theorem}\label{t41}
Let $K_6(n):=K(q,q^{-n},q^{1-n},q^{2-n},q^{3-n},u,v,q^4)$ for $n\geq3$,  and $K_6(0)$ $=K_6(1)=K_6(2)=1$. Then
(i)  the
pair of sequences $(\alpha_n, \beta_n)$ is a Bailey pair with
respect to $a=1$, where $\alpha_0=\beta_0=1$, $\alpha_{4r+2}=0$,
and {\allowdisplaybreaks
\begin{align}\label{chubp41}
\alpha_{4r}&=
\frac{\left(\displaystyle{q^4u,q^4v, q^5/u, q^5/v;q^4}\right)_r}
   {\left(\displaystyle{u,v,q/u,q/v;q^4}\right)_r}q^{8r^2-10r}
   + \frac{\left(\displaystyle{q^3u,q^3v, q^4/u, q^4/v;q^4}\right)_r}
   {\left(\displaystyle{u/q,v/q,1/u,1/v;q^4}\right)_r}q^{8r^2-6r}\\
 \alpha_{4r+1}&=
-\frac{\left(\displaystyle{q^4u,q^4v, q^5/u, q^5/v;q^4}\right)_r}
   {\left(\displaystyle{u,v,q/u,q/v;q^4}\right)_r}q^{8r^2-2r+1}\notag\\
\alpha_{4r-1}&=
-\frac{\left(\displaystyle{q^3u,q^3v, q^4/u, q^4/v;q^4}\right)_r}
   {\left(\displaystyle{u/q,v/q,1/u,1/v;q^4}\right)_r}q^{8r^2-14r+1}\notag\\
\beta_n&= K_6(n) \frac{(-q^2 ;q^2)_{n-1}}{(q;q)_{2n}};\notag
\end{align}
}
(ii)  the pair of sequences $(\alpha_n, \beta_n)$ is a Bailey pair
with respect to $a=1$, where $\alpha_0=\beta_0=1$,
$\alpha_{4r+2}=0$,
and {\allowdisplaybreaks
\begin{align}\label{chubp42}
\alpha_{4r}&=
\frac{\left(\displaystyle{q^4u,q^4v, q^5/u, q^5/v;q^4}\right)_r}
   {\left(\displaystyle{u,v,q/u,q/v;q^4}\right)_r}q^{8r^2-6r}
   + \frac{\left(\displaystyle{q^3u,q^3v, q^4/u, q^4/v;q^4}\right)_r}
   {\left(\displaystyle{u/q,v/q,1/u,1/v;q^4}\right)_r}q^{8r^2-10r}\\
 \alpha_{4r+1}&=
-\frac{\left(\displaystyle{q^4u,q^4v, q^5/u, q^5/v;q^4}\right)_r}
   {\left(\displaystyle{u,v,q/u,q/v;q^4}\right)_r}q^{8r^2-6r}\notag\\
\alpha_{4r-1}&=
-\frac{\left(\displaystyle{q^3u,q^3v, q^4/u, q^4/v;q^4}\right)_r}
   {\left(\displaystyle{u/q,v/q,1/u,1/v;q^4}\right)_r}q^{8r^2-10r}\notag\\
\beta_n&= K_6(n) \frac{q^n\,(-q^2 ;q^2)_{n-1}}{(q;q)_{2n}}.\notag
\end{align}
}
\end{theorem}
\begin{proof} The proof is initially  similar to the proof of Theorem \ref{t23},
except we replace $q$ with $q^4$ and set $b=q^{-n}$,
$c=q^{1-n}$, $d=q^{2-n}$ and $e=q^{3-n}$.  Instead of \eqref{t3eq}, we arrive at
\begin{multline}\label{t41eq}
\sum_{r\in \mathbb{Z}}\frac{1-a q^{8r}}{1-a} \frac{a^{2r}
q^{8r^2-12r}}{(a q;q)_{n+4r}(q;q)_{n-4r}}
\frac{\left(\displaystyle{q^4u,q^4v,aq^4/u,aq^4/v;q^4}\right)_r}
{\left(\displaystyle{u,v,a/u,a/v;q^4}\right)_r}\\
= K(a,q^{-n},q^{1-n},q^{2-n},q^{3-n},u,v,q^4)\frac{(q^4/a,a/q,a,a q;q^4)_{\infty}}{(q,q^2,q^3,a^2/q^2;q^4)_{\infty}} \frac{(-a/q;q^2)_{n}}{(a;q)_{2n}}.
\end{multline}
Next, separate the sum into terms of positive and negative index,
re-index the latter sum by replacing $r$ with $-r$ (and also using
\eqref{binneg}), to get
{\allowdisplaybreaks
\begin{multline}\label{t41eq2} \sum_{r\geq
0}\frac{1-a q^{8r}}{1-a} \frac{a^{2r} q^{8r^2-12r}}{(a
q;q)_{n+4r}(q;q)_{n-4r}}
\frac{\left(\displaystyle{q^4u,q^4v,aq^4/u,aq^4/v;q^4}\right)_r}
{\left(\displaystyle{u,v,a/u,a/v;q^4}\right)_r}\\
+\sum_{r\geq 1}\frac{1-a q^{-8r}}{1-a} \frac{a^{-2r}
q^{8r^2-4r}}{(a q;q)_{n-4r}(q;q)_{n+4r}}
\frac{\left(\displaystyle{q^4/u,q^4/v,uq^4/a,vq^4/a;q^4}\right)_r}
{\left(\displaystyle{1/u,1/v,u/a,v/a;q^4}\right)_r}\\
= K(a,q^{-n},q^{1-n},q^{2-n},q^{3-n},u,v,q^4)\frac{(q^4/a,a/q,a,a q;q^4)_{\infty}}{(q,q^2,q^3,a^2/q^2;q^4)_{\infty}} \frac{(-a/q;q^2)_{n}}{(a;q)_{2n}}.
\end{multline}
}

The proof for the Bailey pair at \eqref{chubp41} then follows, upon
setting $a=q$, simplifying the product side, and using the
identities
\begin{align*}
(1-q^{8r+1})q^{8r^2-10r}&=q^{8r^2-10r}(1-q^{n+4r+1})-q^{8r^2-2r+1}(1-q^{n-4r}),\\
(1-q^{-8r+1})q^{8r^2-6r}&=q^{8r^2-6r}(1-q^{n-4r+1})-q^{8r^2-14r+1}(1-q^{n+4r}).
\end{align*}

The proof for the Bailey pair at \eqref{chubp42} is similar, except we use the identities
\begin{align*}
(1-q^{8r+1})q^{8r^2-10r}&=q^{-n}q^{8r^2-6r}((1-q^{n+4r+1})-(1-q^{n-4r})),\\
(1-q^{-8r+1})q^{8r^2-6r}&=q^{-n}q^{8r^2-10r}((1-q^{n-4r+1})-(1-q^{n+4r})).
\end{align*}

\end{proof}

\begin{theorem}\label{t42}
Let $K_7(n):=K(q^2,q^{-n},q^{1-n},q^{2-n},q^{3-n},u,v,q^4)$ for $n\geq 0$. Then
(i) the
pair of sequences $(\alpha_n, \beta_n)$ is a Bailey pair with
respect to $a=q$, where $\alpha_0=\beta_0=1$,
and {\allowdisplaybreaks
\begin{align}\label{chubp43}
\alpha_{4r}&=
\frac{\left(\displaystyle{q^4u,q^4v, q^6/u, q^6/v;q^4}\right)_r}
   {\left(\displaystyle{u,v,q^2/u,q^2/v;q^4}\right)_r}q^{8r^2-8r}\\
 \alpha_{4r+1}&=
-\frac{\left(\displaystyle{q^4u,q^4v, q^6/u, q^6/v;q^4}\right)_r}
   {\left(\displaystyle{u,v,q^2/u,q^2/v;q^4}\right)_r}q^{8r^2+2}\notag\\
\alpha_{4r-1}&=
\frac{\left(\displaystyle{q^2u,q^2v, q^4/u, q^4/v;q^4}\right)_r}
   {\left(\displaystyle{u/q^2,v/q^2,1/u,1/v;q^4}\right)_r}q^{8r^2-8r}\notag\\
\alpha_{4r-2}&=
-\frac{\left(\displaystyle{q^2u,q^2v, q^4/u, q^4/v;q^4}\right)_r}
   {\left(\displaystyle{u/q^2,v/q^2,1/u,1/v;q^4}\right)_r}q^{8r^2-16r+2}\notag\\
   \beta_n&= K_7(n) \frac{(-q ;q^2)_{n}}{(q^2;q)_{2n}};\notag
\end{align}
}
(ii)  the pair of sequences $(\alpha_n, \beta_n)$ is a Bailey pair
with respect to $a=q$, where $\alpha_0=\beta_0=1$,
and {\allowdisplaybreaks
\begin{align}\label{chubp44}
\alpha_{4r}&=
\frac{\left(\displaystyle{q^4u,q^4v, q^6/u, q^6/v;q^4}\right)_r}
   {\left(\displaystyle{u,v,q^2/u,q^2/v;q^4}\right)_r}q^{8r^2-4r}\\
 \alpha_{4r+1}&=
-\frac{\left(\displaystyle{q^4u,q^4v, q^6/u, q^6/v;q^4}\right)_r}
   {\left(\displaystyle{u,v,q^2/u,q^2/v;q^4}\right)_r}q^{8r^2-4r}\notag\\
\alpha_{4r-1}&=
\frac{\left(\displaystyle{q^2u,q^2v, q^4/u, q^4/v;q^4}\right)_r}
   {\left(\displaystyle{u/q^2,v/q^2,1/u,1/v;q^4}\right)_r}q^{8r^2-12r}\notag\\
\alpha_{4r-2}&=
-\frac{\left(\displaystyle{q^2u,q^2v, q^4/u, q^4/v;q^4}\right)_r}
   {\left(\displaystyle{u/q^2,v/q^2,1/u,1/v;q^4}\right)_r}q^{8r^2-12r}\notag\\
   \beta_n&= K_7(n) \frac{q^n(-q ;q^2)_{n}}{(q^2;q)_{2n}}.\notag
\end{align}
}
\end{theorem}
\begin{proof}
Set $a=q^2$ in \eqref{t41eq2}, simplify the product side and absorb the $1-q^2$ terms on the sum side into what were previously the $(aq;q)_{n+4r}$ and $(aq;q)_{n-4r}$ factors.
The  pair at \eqref{chubp43}  follows after using the identities
\begin{align*}
(1-q^{8r+2})q^{8r^2-8r}&=q^{8r^2-8r}(1-q^{n+4r+2})-q^{8r^2+2}(1-q^{n-4r}),\\
(1-q^{-8r+2})q^{8r^2-8r}&=q^{8r^2-8r}(1-q^{n-4r+2})-q^{8r^2-16r+2}(1-q^{n+4r}).
\end{align*}

The pair at \eqref{chubp44} follows similarly, after employing the identities
\begin{align*}
(1-q^{8r+2})q^{8r^2-8r}&=q^{-n}q^{8r^2-4r}((1-q^{n+4r+2})-(1-q^{n-4r})),\\
(1-q^{-8r+2})q^{8r^2-8r}&=q^{-n}q^{8r^2-12r}((1-q^{n-4r+2})-(1-q^{n+4r})).
\end{align*}
\end{proof}

\begin{theorem}\label{t43}
Let $K_8(0)=1$,
\begin{equation*}
K_8(1)=\frac{(q-u) \left(q^2-u\right) (q-v)
   \left(q^2-v\right)}{\left(q^3-u\right) (u-1) \left(q^3-v\right)
   (v-1)}\\-\frac{(q-1)^4 (q+1)^2 \left(q^2+q+1\right) u v}{q
   \left(q^3-u\right) (u-1) \left(q^3-v\right) (v-1)},
\end{equation*}
and $K_8(n):=K(q^3,q^{-n},q^{1-n},q^{2-n},q^{3-n},u,v,q^4)$ for
$n\geq 2$. Then \\
(i) the pair of sequences $(\alpha_n, \beta_n)$ is a Bailey pair with respect to $a=q^2$, where $\alpha_0=\beta_0=1$,
and {\allowdisplaybreaks
\begin{align}\label{chubp45}
\alpha_{4r}&= \frac{\left(\displaystyle{q^4u,q^4v, q^7/u,
q^7/v;q^4}\right)_r}
   {\left(\displaystyle{u,v,q^3/u,q^3/v;q^4}\right)_r}\,q^{8r^2-6r}\\
 \alpha_{4r+1}&=
 -\frac{\left(\displaystyle{q^4u,q^4v, q^7/u,
q^7/v;q^4}\right)_r}
   {\left(\displaystyle{u,v,q^3/u,q^3/v;q^4}\right)_{r}}\,q^{8r^2+2r+3}
-\frac{\left(\displaystyle{q^4/u,q^4/v, q u, q
v;q^4}\right)_{r+1}}
   {\left(\displaystyle{1/u,1/v,u/q^3,v/q^3;q^4}\right)_{r+1}}\,q^{8r^2-2r-7}\notag\\
\alpha_{4r-1}&= 0\notag\\
\alpha_{4r-2}&= \frac{\left(\displaystyle{q^4/u,q^4/v, q u, q
v;q^4}\right)_{r}}
   {\left(\displaystyle{1/u,1/v,u/q^3,v/q^3;q^4}\right)_{r}}\,q^{8r^2-10r}\notag\\
   \beta_n&= K_8(n) \frac{(-q^2 ;q^2)_{n}}{(q^3;q)_{2n}};\notag
\end{align}
} (ii)  the pair of sequences $(\alpha_n, \beta_n)$ is a Bailey
pair with respect to $a=q^2$, where $\alpha_0=\beta_0=1$, and
{\allowdisplaybreaks
\begin{align}\label{chubp46}
\alpha_{4r}&= \frac{\left(\displaystyle{q^4u,q^4v, q^7/u,
q^7/v;q^4}\right)_r}
   {\left(\displaystyle{u,v,q^3/u,q^3/v;q^4}\right)_r}\,q^{8r^2-2r}\\
 \alpha_{4r+1}&=
 -\frac{\left(\displaystyle{q^4u,q^4v, q^7/u,
q^7/v;q^4}\right)_r}
   {\left(\displaystyle{u,v,q^3/u,q^3/v;q^4}\right)_{r}}\,q^{8r^2-2r}
-\frac{\left(\displaystyle{q^4/u,q^4/v, q u, q
v;q^4}\right)_{r+1}}
   {\left(\displaystyle{1/u,1/v,u/q^3,v/q^3;q^4}\right)_{r+1}}\,q^{8r^2+2r-6}\notag\\
\alpha_{4r-1}&= 0\notag\\
\alpha_{4r-2}&= \frac{\left(\displaystyle{q^4/u,q^4/v, q u, q
v;q^4}\right)_{r}}
   {\left(\displaystyle{1/u,1/v,u/q^3,v/q^3;q^4}\right)_{r}}\,q^{8r^2-14r}\notag\\
   \beta_n&= K_8(n) q^n \frac{(-q^2 ;q^2)_{n}}{(q^3;q)_{2n}};\notag
\end{align}
}
\end{theorem}
\begin{remark} The reason $K_8(1)$ does not fit into the formula for  $K_8(n)$ for $n>1$ is that it is necessary to be careful with the order in which (after replacing $q$ with $q^4$) we set $a=q^3$, $b=1/q$, $c=1$, $d=q$ and $e=q^2$, in order to get the series at \eqref{chueq1} to converge. It is easy to see that making those substitutions simultaneously gives $a^2/q^4bcde=1$, contradicting the requirement that $|a^2/q^4bcde|<1$ in the series. One way around this is to first set $c=1$ and $b=a/q^4$, causing the series to terminate above and below, after which the other replacements can be made.

Note also that, with the stated choices for the parameters,
the $a^2-bcdeq^4$ factor in the denominator of the expression for $K$ at \eqref{Keq} vanishes, and in fact these choices also cause the numerator of $K$ to vanish. It is cancelling these ``zero factors" that make the expression for $K_8(1)$ different.    The order of substitutions described above causes $K(q^3,q^{-1},1,q,q^{2},u,v,q^4)$ to have the value assigned to $K_8(1)$, whereas for $n>1$, $K(q^3,q^{-n},q^{1-n},q^{2-n},q^{3-n},$ $u,v,q^4)$ is independent of the order of substitutions.

We also note that a similar situation occurs in some of the other theorems above.
\end{remark}

\begin{proof}
The proof is similar to the proof of Theorem \ref{t42}, except we
set $a=q^3$ in \eqref{t41eq2}. The pair at \eqref{chubp45}  follows
after using the identities
\begin{align*}
(1-q^{8r+3})q^{8r^2-6r}&=q^{8r^2-6r}(1-q^{n+4r+3})-q^{8r^2+2r+3}(1-q^{n-4r}),\\
(1-q^{-8r+3})q^{8r^2-10r}&=q^{8r^2-10r}(1-q^{n-4r+3})-q^{8r^2-18r+3}(1-q^{n+4r}),
\end{align*}
and then replacing $r$ with $r+1$ in the sum corresponding to the
second term in the second identity above (so that this sum
implicitly defines part of $\alpha_{4r+1}$ instead of
$\alpha_{4r-3}$). The pair at \eqref{chubp46} follows similarly,
after using the identities {\allowdisplaybreaks
\begin{align*}
(1-q^{8r+3})q^{8r^2-6r}&=q^{-n}q^{8r^2-2r}((1-q^{n+4r+3})-(1-q^{n-4r})),\\
(1-q^{-8r+3})q^{8r^2-10r}&=q^{-n}q^{8r^2-14r}((1-q^{n-4r+3})-(1-q^{n+4r})).
\end{align*}
}
\end{proof}

The six Bailey pairs in Slater's \textbf{K} table \cite[page 471]{S51} follow from the three theorems in this section, upon letting $u, v \to \infty$.
\begin{corollary}
The following pairs of sequences $(\alpha_n, \beta_n)$, with $\alpha_0=\beta_0=1$, are Bailey pairs with respect to the stated value of $a$.

{\allowdisplaybreaks
\begin{align}\label{chubp41a}
\alpha_{4r}&=
q^{8r^2-2r}+ q^{8r^2+2r}\\
 \alpha_{4r+1}&=
-q^{8r^2+6r+1}\notag\\
\alpha_{4r-1}&=
-q^{8r^2-6r+1}\notag\\
   \alpha_{4r-2}&=0 \notag\\
\beta_n&= \frac{(-q^2 ;q^2)_{n-1}}{(q;q)_{2n}}\text{ with respect to $a=1$};\notag
\end{align}
} {\allowdisplaybreaks
\begin{align}\label{chubp42a}
\alpha_{4r}&=
q^{8r^2+2r}+ q^{8r^2-2r}\\
 \alpha_{4r+1}&=
-q^{8r^2+2r}\notag\\
\alpha_{4r-1}&=
-q^{8r^2-2r}\notag\\
\alpha_{4r-2}&=0 \notag\\
\beta_n&= \frac{q^n\,(-q^2 ;q^2)_{n-1}}{(q;q)_{2n}}\text{ with respect to $a=1$};\notag
\end{align}
}
 {\allowdisplaybreaks
\begin{align}\label{chubp43a}
\alpha_{4r}&=
q^{8r^2}\\
 \alpha_{4r+1}&=
-q^{8r^2+8r+2}\notag\\
\alpha_{4r-1}&=
q^{8r^2}\notag\\
\alpha_{4r-2}&=
-q^{8r^2-8r+2}\notag\\
   \beta_n&=  \frac{(-q ;q^2)_{n}}{(q^2;q)_{2n}}\text{ with respect to $a=q$};\notag
\end{align}
}
\begin{align}\label{chubp44a}
\alpha_{4r}&=
q^{8r^2+4r}\\
 \alpha_{4r+1}&=
-q^{8r^2+4r}\notag\\
\alpha_{4r-1}&=
q^{8r^2-4r}\notag\\
\alpha_{4r-2}&=
-q^{8r^2-4r}\notag\\
   \beta_n&=  \frac{q^n(-q ;q^2)_{n}}{(q^2;q)_{2n}}\text{ with respect to $a=q$};\notag
\end{align}
{\allowdisplaybreaks
\begin{align}\label{chubp45a}
\alpha_{4r}&= q^{8r^2+2r}\\
 \alpha_{4r+1}&=
 -q^{8r^2+10r+3}-q^{8r^2+6r+1}\notag\\
\alpha_{4r-1}&= 0\notag\\
\alpha_{4r-2}&= q^{8r^2-2r}\notag\\
   \beta_n&= \frac{(-q^2 ;q^2)_{n}}{(q^3;q)_{2n}}\text{ with respect to $a=q^2$};\notag
\end{align}
} {\allowdisplaybreaks
\begin{align}\label{chubp46a}
\alpha_{4r}&= q^{8r^2+6r}\\
 \alpha_{4r+1}&=
 -q^{8r^2+6r}-q^{8r^2+10r+2}\notag\\
\alpha_{4r-1}&= 0\notag\\
\alpha_{4r-2}&= q^{8r^2-6r}\notag\\
   \beta_n&=  q^n \frac{(-q^2 ;q^2)_{n}}{(q^3;q)_{2n}}\text{ with respect to $a=q^2$}.\notag
\end{align}
}
\end{corollary}

\section{Identities of the Rogers-Ramanujan-Slater type
and False Theta Series Identities}\label{secids}

We use the Bailey pairs found in the previous section to derive
several new series-product identities, and also to give new proofs of
some general identities due to Ramanujan, Andrews and others. We also remark that we sometimes use a more general case of
the $_6\psi_6$ than the Jacobi Triple Product Identity, in contrast
to how Slater derived her products.

\subsection{General identities containing one or more free parameters}
We first give new proofs for two
quite general identities of Ramanujan (\cite[page 33]{R88}, see also
R1 and R2 on page 8 of \cite{McLSZ08}). One reason these two identities are of interest is that each leads to infinitely many identities of Rogers-Ramanujan type (set $z=\pm q^{a/b}$, where $a$ and $b>0$ are integers, and then replace $q$ with $q^b$).

\begin{theorem}\label{r1}
For $|q|<1$ and $z\in \mathbb{C}\setminus\{ 0 \}$,
\begin{equation}\label{geneqr1}
\sum_{n=0}^{\infty}\frac{(q/z;q)_{n}(z;q)_nq^{n^2}}{(q;q)_{2n}}=
\frac{(qz,q^2/z,q^3;q^3)_{\infty}}{(q;q)_{\infty}}.
\end{equation}
\end{theorem}

\begin{proof}
Let $e \to 0$ in the Bailey pair at \eqref{sgen35}, and then insert
the resulting pair into \eqref{Seq1}, after setting $y=q/z$. This
leads to the identities {\allowdisplaybreaks
\begin{align*}
&\sum_{n=0}^{\infty}\frac{\left(\frac{q}{z};q\right)_{n}(z;q)_nq^{n^2}}{(q;q)_{2n}}=
\frac{\left(q z,\frac{q^2}{z};q\right)_{\infty}}{(q^2,q;q)_{\infty}}
\bigg[\sum_{r=0}^{\infty}\frac{1-q^{6r+1}}{1-q}\frac{\left(\frac{q}{z},z;q\right)_{3r}q^{3r^2+r}}
{\left(qz,\frac{q^2}{z};q\right)_{3r}}
-
\sum_{r=1}^{\infty}\frac{1-q^{6r-1}}{1-q}\frac{(q/z,z;q)_{3r-1}q^{3r^2-r}}{(qz,q^2/z;q)_{3r-1}}
\bigg ]\\
&= \frac{\left(q
z,\frac{q^2}{z};q\right)_{\infty}}{(q^2,q;q)_{\infty}}
\bigg[\sum_{r=0}^{\infty}\frac{1-q^{6r+1}}{1-q}\frac{\left(\frac{q}{z},z;q^3\right)_{r}q^{3r^2+r}}
{\left(q^3z,\frac{q^4}{z};q^3\right)_{r}}
+\sum_{r=1}^{\infty}\frac{1-q^{6r-1}}{1-1/q}\frac{(1/z,z/q;q^3)_{r}q^{3r^2-r}}{(q^2z,q^3/z;q^3)_{r}}
\bigg ]\\
&=\frac{\left(q
z,\frac{q^2}{z};q\right)_{\infty}}{(q^2,q;q)_{\infty}}\frac{(q^4,q^3,q^3,q^2;q^3)_{\infty}}
{(q^4/z,zq^3,q^3/z,q^2z;q^3)_{\infty}}.
\end{align*}
} The second equality above follows after some elementary
$q$-product manipulations and the last equality follows from the
fact that the two series in the previous equality combine to give a
special case of Bailey's $_6\psi_6$ summation  at \eqref{baileyeq1}
(replace $q$ with $q^3$, let $d, e \to \infty$, set $a=q$, $b=q/z$
and $c=z$). The result now follows after some further elementary
manipulations.
\end{proof}

\begin{theorem}\label{r2}
For $|q|<1$ and $z\in \mathbb{C}\setminus\{ 0\}$,
\begin{equation}\label{geneqr2}
\sum_{n=0}^{\infty}\frac{(q^2/z;q^2)_{n}(z;q^2)_nq^{n^2}}{(q;q^2)_n(q^4;q^4)_{n}}=
\frac{(-q;q^2)_{\infty}}{(q^2;q^2)_{\infty}}\,(qz,q^3/z,q^4;q^4)_{\infty}.
\end{equation}
\end{theorem}

\begin{proof}The proof is similar to that of the theorem above.
Let $e \to 0$ and set $d=q^{3/2}$ in the Bailey pair at \eqref{Sgen7}, and then insert
the resulting pair into \eqref{Seq1}, after setting $y=q/z$. This
leads to the identities
{\allowdisplaybreaks
\begin{align*}
&\sum_{n=0}^{\infty}\frac{\left(\frac{q}{z};q\right)_{n}(z;q)_nq^{n^2/2}}{(q^{1/2};q)_n(q^2;q^2)_{n}}\\&=
\frac{\left(q z,\frac{q^2}{z};q\right)_{\infty}}{(q^2,q;q)_{\infty}}
\bigg[\sum_{r=0}^{\infty}\frac{1-q^{4r+1}}{1-q}\frac{\left(\frac{q}{z},z;q\right)_{2r}
q^{r^2+r/2}} {\left(qz,\frac{q^2}{z};q\right)_{2r}(-1)^r}
-
\sum_{r=1}^{\infty}\frac{1-q^{4r-1}}{1-q}\frac{(q/z,z;q)_{2r-1}q^{r^2-r/2}(-1)^r}{(qz,q^2/z;q)_{2r-1}}
\bigg ]\\
&= \frac{\left(q
z,\frac{q^2}{z};q\right)_{\infty}}{(q^2,q;q)_{\infty}}
\bigg[\sum_{r=0}^{\infty}\frac{1-q^{4r+1}}{1-q}\frac{\left(\frac{q}{z},z;q^2\right)_{r}(-1)^rq^{r^2+r/2}}
{\left(q^2z,\frac{q^3}{z};q^2\right)_{r}}
+\sum_{r=1}^{\infty}\frac{1-q^{4r-1}}{1-1/q}\frac{(1/z,z/q;q^2)_{r}(-1)^rq^{r^2-r/2}}{(qz,q^2/z;q^2)_{r}}
\bigg ]\\
&=\frac{\left(q
z,\frac{q^2}{z};q\right)_{\infty}}{(q^2,q;q)_{\infty}}\frac{(q^3,q^2,z q^{1/2},q^{3/2}/z,q^2,q;q^2)_{\infty}}
{(q^3/z,zq^2,q^{3/2},q^2/z,qz,q^{1/2};q^2)_{\infty}}.
\end{align*}
}
The second equality above once again follows after some elementary $q$-product
manipulations and once again the
two series on the right side in the second equality combine to give a special case
of Bailey's $_6\psi_6$ summation  at \eqref{baileyeq1} (replace $q$ with
$q^2$, let $e \to \infty$, set $a=q$, $b=q/z$, $c=z$ and $d=q^{3/2}$). The
result now follows after replacing $q$ with $q^2$, followed by some further elementary manipulations.
\end{proof}

Different proofs of the results in the two theorems above were also
given in \cite{AB09} (Entries \textbf{5.3.1} and \textbf{5.3.5}) and
in \cite{P88}.

We now prove a new general series-product identity, one which may be regarded as a partner to Ramanujan's result in Theorem \ref{r1}.

\begin{theorem}\label{r1comp}
For $|q|<1$ and $z\in \mathbb{C}\setminus \{ 0 \}$,
\begin{equation}\label{geneq1}
\sum_{n=0}^{\infty}\frac{(q/z;q)_{n+1}(z;q)_nq^{n^2+n}}{(q;q)_{2n+1}}=
\frac{(q^2z,q/z,q^3;q^3)_{\infty}}{(q;q)_{\infty}}.
\end{equation}
\end{theorem}

\begin{proof}
Let $e \to 0$ in the Bailey pair at \eqref{sgen37}, and then insert
the resulting pair into \eqref{Seq1}, after setting $y=q^2/z$. This
leads to the identities {\allowdisplaybreaks
\begin{align*}
&\sum_{n=0}^{\infty}\frac{\left(\frac{q^2}{z};q\right)_{n}(z;q)_nq^{n^2+n}}{(q^2;q)_{2n}}\\
&= \frac{\left(q
z,\frac{q^3}{z};q\right)_{\infty}}{(q^3,q;q)_{\infty}}
\bigg[\sum_{r=0}^{\infty}\frac{1-q^{6r+2}}{1-q^2}\frac{\left(\frac{q^2}{z},z;q\right)_{3r}q^{3r^2+2r}}
{\left(qz,\frac{q^3}{z};q\right)_{3r}}
-\sum_{r=0}^{\infty}\frac{1-q^{6r+4}}{1-q^2}\frac{(q^2/z,z;q)_{3r+1}q^{3r^2+4r+1}}{(qz,q^3/z;q)_{3r+1}}
\bigg ]\\
&= \frac{\left(q
z,\frac{q^3}{z};q\right)_{\infty}}{(q^3,q;q)_{\infty}}
\bigg[\sum_{r=0}^{\infty}\frac{1-q^{6r+2}}{1-q^2}\frac{\left(\frac{q^2}{z},z;q^3\right)_{r}q^{3r^2+2r}}
{\left(q^3z,\frac{q^5}{z};q^3\right)_{r}}
+\sum_{r=1}^{\infty}\frac{1-q^{6r-2}}{1-1/q^2}\frac{(1/z,z/q^2;q^3)_{r}q^{3r^2-2r}}{(qz,q^3/z;q^3)_{r}}
\bigg ]\\
&=\frac{\left(q
z,\frac{q^3}{z};q\right)_{\infty}}{(q^3,q;q)_{\infty}}\frac{(q^5,q^3,q^3,q;q^3)_{\infty}}
{(q^5/z,zq^3,q^3/z,qz;q^3)_{\infty}}.
\end{align*}
} The second equality above follows after some elementary
$q$-product manipulations and re-indexing the second series
(replacing $r$ with $r-1$). The last equality follows from the fact
that the two series in the previous equality combine to give another
special case of Bailey's $_6\psi_6$ summation \eqref{baileyeq1}
(replace $q$ with $q^3$, let $d, e \to \infty$, set $a=q^2$,
$b=q^2/z$ and $c=z$). The result once again follows after some further
elementary manipulations.
\end{proof}

\begin{remark}
Several identities on Slater's list \cite{S52} follow
as special cases of~\eqref{geneq1}.
We summarize these in the following table.

\begin{center}
\begin{tabular}{|c| c | c |}
\hline
Replace $q$ by & set $z$ to & to obtain Slater's identity\\
\hline\hline
$q$ & $-q$ & (22) \\
\hline
$q^2$ & $-q$ & (27) = (87) \\
\hline
$q$  & $iq$ & (28) \\
\hline
$q^3$ & $q^2$ & (40)\\
\hline
$q^3$ & $q$ & (41)\\
\hline
$q^4$ & $q^3$ & (55) = (57) \\
\hline
$q$ & $e^{2\pi i /3} q$ & (92)\\
\hline
\end{tabular}
\end{center}
\end{remark}

We next give new proofs of Andrews'  $q$-analog of Gauss's $_2F_1(
1/ 2)$ sum and a special case ($b \to \infty$) of Heine's~\cite{H47}
$q$-analog of Gauss's sum:
\begin{equation}
\sum_{n=0}^{\infty}\frac{(a,b\,;q)_n}{(c,q\,;q)_n}\left(\frac{c}{ab}\right)^n
=\frac{(c/a,c/b \,;q)_{\infty}}{(c,c/ab \,;q)_{\infty}}. 
\end{equation}

We first recall Jackson's summation formula for a very-well-poised $_6\phi_5$
series~\cite[p. 356, Eq. (II. 20)]{GR04} (which follows upon setting $e=a$ in \eqref{baileyeq1}):
\begin{equation}\label{6phi5}
\sum_{n=0}^{\infty}\frac{(a,q\sqrt{a},-q\sqrt{a},b,c,d;q)_n}
{\left(q,\sqrt{a},-\sqrt{a},\frac{aq}{b},\frac{aq}{c},\frac{aq}{d};q\right)_n}
\left(\frac{aq}{bcd}\right)^n=\frac{(aq,aq/bc,aq/bd,aq/cd;q)_\infty}{(aq/b,aq/c,aq/d,aq/bcd;q)_\infty}.
\end{equation}

\begin{theorem}
\begin{equation}\label{aqg}
\sum_{n=0}^{\infty}\frac{(a,b;q)_nq^{n(n+1)/2}}{(q;q)_n(a b q;q^2)_n}
=\frac{(a q,b q;q^2)_{\infty}}{(q,a b q;q^2)_{\infty}}, \hspace{15pt} \text{$($Andrews, \cite{A73}$)$}
\end{equation}
\begin{equation}\label{hqg}
\sum_{n=0}^{\infty}\frac{(a;q)_nq^{n(n-1)/2}(-c)^n}{(c,q;q)_n a^n}
=\frac{(c/a;q)_{\infty}}{(c;q)_{\infty}}, \hspace{15pt} \text{$($Heine, \cite{H47}$)$}
\end{equation}
\end{theorem}

\begin{proof}
Let $d\to 0$ in the Bailey pair at \eqref{sgen660} to get the pair
{\allowdisplaybreaks
\begin{align}\label{sgen660new}
\alpha_{2r}&=
\frac{\displaystyle{1-a q^{4r}}}{\displaystyle{1-a}}
\frac{\left(\displaystyle{a;q^2}\right)_r (-1)^r\displaystyle{q^{r^2-r}}}
   {\left(\displaystyle{q^2;q^2}\right)_r},\\
\alpha_{2r-1}&=0,\notag
\\
\beta_n&= \frac{q^{n(n-1)/2}}{(a q;q^2)_{n}(q;q)_n},\,\, \text{ with respect to $a=a$.}\notag
\end{align}
} Substitute this pair into  \eqref{Seq1}, after setting $y=a/z$, to
get
\begin{align*}
\sum_{n=0}^{\infty}\frac{(a/z,z;q)_nq^{n(n+1)/2}}{(a q;q^2)_{n}(q;q)_n}&
=\frac{(zq,aq/z;q)_{\infty}}{(aq,q;q)_{\infty}}
\sum_{r=0}^{\infty}\frac{1-aq^{4r}}{1-a}
\frac{(a,a/z,z;q^2)_r}{(zq^2,aq^2/z,q^2;q^2)_r}(-1)^r q^{r^2+r}\notag \\
&
=\frac{(zq,aq/z;q)_{\infty}}{(aq,q;q)_{\infty}}\frac{(aq^2,q^2;q^2)_{\infty}}{(q^2z,aq^2/z;q^2)_{\infty}}\notag \\
&=\frac{(zq,aq/z;q^2)_{\infty}}{(aq,q;q^2)_{\infty}}.
\end{align*}
The next-to-last equality follows from \eqref{6phi5} (replace $q$ with $q^2$, let $d \to \infty$, set $b=a/z$ and $c=z$), and \eqref{aqg} follows upon replacing $a$ with $a b$ and $z$ with $b$.

Next, substitute the pair at \eqref{sgen660new} once again into \eqref{Seq1}, this time setting $y=\sqrt{aq}$.
This gives
{\allowdisplaybreaks
\begin{align*}
\sum_{n=0}^{\infty}\frac{(z;q)_nq^{n(n-1)/2}}{(-\sqrt{a q};q)_{n}(q;q)_n}\left(\frac{\sqrt{aq}}{z}\right)^n&
=\frac{(\sqrt{aq},aq/z;q)_{\infty}}{(aq,\sqrt{aq}/z;q)_{\infty}}
 \sum_{r=0}^{\infty}\frac{1-aq^{4r}}{1-a}
\frac{(a,z,zq;q^2)_r q^{r^2+r}}{(aq^2/z,aq^3/z,q^2;q^2)_r} \left(\frac{-a}{z^2q}\right)^r\notag \\
&
=\frac{(\sqrt{aq},aq/z;q)_{\infty}}{(aq,\sqrt{aq}/z;q)_{\infty}}
\frac{(aq^2,aq/z^2;q^2)_{\infty}}{(aq^2/z,aq/z;q^2)_{\infty}}\notag \\
&=\frac{(-\sqrt{aq}/z;q)_{\infty}}{(-\sqrt{aq};q)_{\infty}}.
\end{align*}
}
The next-to-last equality follows once again from \eqref{6phi5} (replace $q$ with $q^2$, let $d \to \infty$, set $b=z$ and $c=z q$), and \eqref{hqg} follows upon replacing $-\sqrt{aq}$ with $c$ and $z$ with $a$.
\end{proof}

\subsection{A particular case of the Bailey Transform, I}
If we set $y=-\sqrt{aq}$ and $z=\sqrt{aq}$ in \eqref{Seq1}, the following identity results, providing both series converge:
\begin{equation}\label{Seq1aq}
\sum_{n=0}^{\infty}(aq;q^2)_n \left (-1\right)^n \beta_n
=\frac{1}{(a q^2;q^2)_{\infty}(-1;q)_{\infty}}
\sum_{n=0}^{\infty}\left (-1\right)^n \alpha_n.
\end{equation}

This particular case was not used by Slater \cite{S52}, and may possibly be regarded as being of lesser importance for two reasons. Firstly, it is certainly the case the both series above will not converge for all Bailey pairs $(\alpha_n,\beta_n)$, and secondly, even if a series-product identity does result, the power of $q$ on the series side may not be quadratic in the exponent (and many do not consider such identities as being of Rogers-Ramanujan-Slater type), or else the resulting identity is an easy consequence of a more general identity. However we believe such identities are sufficiently interesting to include some examples.
\begin{corollary}\label{cor3.6}
\begin{equation}\label{eqaq1}
1+\sum_{n=1}^{\infty}\frac{(-q^2;q^2)_{n-1}q^n}{(q^2;q^2)_n}
=\frac{(-q;q^2)_{\infty}}{(q^2;q^2)_{\infty}}(-q^6,-q^{10},q^{16};q^{16})_{\infty},
\end{equation}
\begin{equation}\label{eqaq2}
\sum_{n=1}^{\infty}\frac{(q;q^2)_{n}q^n}{(-q;q^2)_{n+1}}
=1+\sum_{r=1}^{\infty}q^{8r^2}(q^{4r}-q^{-4r}),
\end{equation}
\begin{equation}\label{eqaq3}
\sum_{n=0}^{\infty}\frac{(-q^2;q^2)_{n}q^n}{(q^2;q^2)_{n+1}}
=\frac{(-q;q^2)_{\infty}}{(q^2;q^2)_{\infty}}(-q^2,-q^{14},q^{16};q^{16})_{\infty},
\end{equation}
\begin{equation}\label{eqaq4}
\sum_{n=0}^{\infty}\frac{(q^3;q^3)_{n}(-q)^n}{(q^2;q^2)_{n+1}(q;q)_n}
=\frac{(q;q^2)_{\infty}}{(q^2;q^2)_{\infty}}\frac{(q^{18};q^{18})_{\infty}}{(q^{9};q^{18})_{\infty}}.
\end{equation}
\end{corollary}
\begin{proof}
For the first three identities above, insert the Bailey pairs at \eqref{chubp42a}, \eqref{chubp44a} and
\eqref{chubp46a}, respectively into \eqref{Seq1aq} (noting that $a=1, q, q^2$, respectively), and in each case the result follows after some elementary manipulation of the resulting identity.

For \eqref{eqaq4}, substitute the Bailey pair at \eqref{sgen22} into \eqref{Seq1aq}, set $a=q^2$, and simplify the resulting identity.
\end{proof}

\subsection{A particular case of the Bailey Transform, II}
If we set $y=q\sqrt{a}$ and let $z \to \infty$ in \eqref{Seq1}, the
following identity results:
\begin{equation}\label{Seq1aqf}
\sum_{n=0}^{\infty}(q\sqrt{a};q)_n \left (-\sqrt{a}\right)^n
q^{n(n-1)/2}\beta_n 
=\frac{(q\sqrt{a};q)_{\infty}}{(a
q;q)_{\infty}} \sum_{n=0}^{\infty}(1-\sqrt{a} q^n)\left
(-\sqrt{a}\right)^nq^{n(n-1)/2} \alpha_n.
\end{equation}

This transformation also gives rise to a number of interesting false theta series
identities, and we give several examples below.  See \S\ref{hybrid} below for
an explanation of false theta series.
\begin{corollary}
\begin{equation}\label{Seq1aqfeq1}
\sum_{n=0}^{\infty}\frac{(1-q^{n+1})(q^3;q^3)_n(-1)^n
q^{n(n+1)/2}}{(q;q)_{2n+2}}
=\sum_{r=0}^\infty q^{9r^2 + 3r} (1-q^{12r+6}) +\sum_{r=0}^\infty
q^{9r^2 + 12r + 4} (1-q^{-6r-3} )
\end{equation}
\begin{equation}\label{Seq1aqfeq2}
\sum_{n=0}^{\infty}\frac{(-1)^n q^{n^2+n}}{(-q;q)_{2n}} 
=\sum_{r=0}^\infty q^{10r^2+r} (1-q^{18r+9}) + \sum_{r=0}^\infty
q^{10r^2+11r+3} (1-q^{-2r-1}).
\end{equation}
\begin{multline}\label{Seq1aqfeq4}
\sum_{n=0}^{\infty}\frac{(q;q)_{n+1}(-q^2;q^2)_{n}(-1)^n
q^{(n^2+n)/2}}{(q;q)_{2n+2}}
= \sum_{r=0}^\infty q^{16r^2+4r} (1-q^{24r+12}) \\
-\sum_{r=0}^\infty q^{16r^2+8r+1} (1-q^{16r+8})
+\sum_{r=0}^\infty q^{16r^2+12r+2} (1-q^{8r+4}).
\end{multline}
\begin{equation}\label{Seq1aqfeq5}
\sum_{n=0}^{\infty}\frac{(-q;q)_{n}
q^{n(n+1)/2}(-1)^n}{(q;q^2)_{n+1}}
=\sum_{r=0}^\infty (-1)^r q^{r^2+r}.
\end{equation}
\end{corollary}

\begin{proof}
Let $a=q^2$ in the Bailey pair at \eqref{sgen21}, substitute the
resulting pair into \eqref{Seq1aqf} (with $a=q^2$) and
\eqref{Seq1aqfeq1} follows after simplifying the left side and
re-indexing two of the resulting series on the right side.

For \eqref{Seq1aqfeq2}, let $e\to \infty$ and set $d=q^{3/2}$ in the
Bailey pair at \eqref{Sgen9}, substitute the resulting pair
(Slater's pair \textbf{I(4)}) into \eqref{Seq1aqf} (with $a=1$), replace $q$ with $q^2$ and
simplify.

For \eqref{Seq1aqfeq4},  substitute  pair \eqref{chubp41a} into \eqref{Seq1aqf} (with $a=1$),
 and rearrange.

To get \eqref{Seq1aqfeq5}, let $a \to 1$ and set $c=q^{1/2}$ and
$b=-q^{1/2}$ in Slater's general Bailey pair at  \eqref{Sgen1},
substitute the resulting pair (Slater's pair \textbf{H(1)}, corrected)
into \eqref{Seq1aqf} (again with $a=1$) and rearrange the resulting
identity.


\end{proof}

\begin{remark}
The transformation at \eqref{Seq1aqf} does lead to identities of Rogers-Ramanujan type,
but those we found were either not new, or were simple linear combinations of existing identities.
\end{remark}

\subsection{Miscellaneous Identities of the Rogers-Ramanujan-Slater type} Finally, we exhibit some identities that follow from Bailey pairs not listed by Slater, pairs that do follow however from specializing the parameters in our general Bailey pairs.
We  use two  cases of the transformation at \eqref{Seq1} which were used by Slater. Firstly, let $y,z \to \infty$ to get
\begin{equation}\label{Seq1a}
\sum_{n=0}^{\infty}a^n q^{n^2} \beta_n(a,q)
=\frac{1}{(aq;q)_{\infty}}
\sum_{n=0}^{\infty}a^n q^{n^2} \alpha_n(a,q).
\end{equation}
Secondly, let $z \to \infty$, set $y=-\sqrt{aq}$ and then replace $a$ with $a^2$ and $q$ with $q^2$ to get
\begin{equation}\label{Seq1b}
\sum_{n=0}^{\infty}(-aq;q^2)_n a^n q^{n^2} \beta_n(a^2,q^2)
=\frac{(-aq;q^2)_{\infty}}{(a^2q^2;q^2)_{\infty}}
\sum_{n=0}^{\infty}a^n q^{n^2} \alpha_n(a^2,q^2).
\end{equation}

\begin{corollary}
\begin{equation}\label{misc1eq1}
1+\sum_{n=1}^{\infty}\frac{q^{n^2}(-1)^n}{(1+q^{2n-1})(-1;q)_{n-1}(q;q)_n}
=\frac{1+(q^{4};q^{4})_{\infty}}{(-1;q)_{\infty}}.
\end{equation}
\begin{equation}\label{misc1eq2}
1+\sum_{n=1}^{\infty}\frac{(q;-q^2)_n q^{n^2}(-1)^{\lfloor (n+1)/2 \rfloor}}{(1-q^{4n-2})(-1;-q^2)_{n-1}(-q^2;-q^2)_n}
=\frac{(q,-q^{3};q^{4})_{\infty}}{(-1,q^2;q^4)_{\infty}}\left[1+(-q^4,-q^{12},q^{16};q^{16}) \right].
\end{equation}
\begin{equation}\label{misc1eq5}
\sum_{n=0}^{\infty}\frac{q^{(n^2+n)/2}(-1)^n}{(1-q^{2n+1})(q;q)_n}
=(-q^3,-q^5,q^8;q^8)_{\infty}.
\end{equation}
\begin{equation}\label{misc1eq9}
\sum_{n=0}^{\infty}\frac{q^{(n^2+3n)/2}(-1)^n}{(1-q^{2n+1})(q;q)_n}
=(-q,-q^7,q^8;q^8)_{\infty}.
\end{equation}
\begin{multline}\label{mod4ideq1}
\frac{1+q^3}{(1-q)(1-q^2)}+\frac{1-q^4}{1-q}
\sum_{n=1}^{\infty}\frac{q^{n^2+3n}(-q^2;q^2)_{n-1}}{(q;q)_{2n+2}}
=\frac{1}{(q;q)_{\infty}}\\
\times \bigg [(-q^{18},-q^{30},q^{48};q^{48})_{\infty}+q^3
(-q^{10},-q^{38},q^{48};q^{48})_{\infty}\\-q^6
(-q^{2},-q^{46},q^{48};q^{48})_{\infty}-q^3
(-q^{6},-q^{42},q^{48};q^{48})_{\infty}\bigg ]
\end{multline}
\end{corollary}

\begin{proof}
Let $d=q^2$ and $a=-1$ in the Bailey pair at \eqref{sgen660}.
The identity at \eqref{misc1eq1} follows upon substituting the resulting pair into \eqref{Seq1a} and simplifying.

For \eqref{misc1eq2}, set $d=q^2$ and $a=i$ in the pair at \eqref{sgen660},
substitute the resulting pair into  \eqref{Seq1b}, replace $q$ with $iq$, and rearrange.

If we  set $a=-q$ and  $d=q^2$ in the pair at \eqref{sgen661}, we
get \eqref{misc1eq5} when this pair is inserted in \eqref{Seq1b},
after replacing $q$ with $q^{1/2}$.

 The identity at \eqref{misc1eq9}  follows upon inserting the pair formed
 by setting $a=-q$ and  $d=q^2$ in the Bailey pair at  \eqref{sgen662}
 into \eqref{Seq1b}, and then once again
replacing $q$ with $q^{1/2}$.

Let $v\to \infty$ and set $u=iq^{3/2}$ in the Bailey pair at
\eqref{chubp45}. This gives the Bailey pair (with respect to
$a=q^2$), $\alpha_0=\beta_0=1$, {\allowdisplaybreaks
\begin{align*}
\alpha_{4r}&= \frac{1+q^{8r+3}}{1+q^3}q^{8r^2-2r}\\
 \alpha_{4r+1}&=
 -\frac{1+q^{8r+3}}{1+q^3}q^{8r^2+6r+3}
 -\frac{1+q^{8r+5}}{1+q^3}q^{8r^2+2r}\notag\\
\alpha_{4r-1}&= 0\notag\\
\alpha_{4r-2}&= \frac{1+q^{8r-3}}{1+q^3}q^{8r^2-6r+3}\notag\\
   \beta_n&= \frac{q^n(1+q)(1+q^2)(-q^2 ;q^2)_{n-1}}{(1+q^3)(q^3;q)_{2n}}.\notag
\end{align*}
} Note that the expression for $\beta_n$ follows from the fact that,
with the stated values for $u$ and $v$,
\[
K_8(n)=\frac{(1+q)\left(1+q^2\right) q^n}{\left(1+q^3\right)
\left(1+q^{2
   n}\right)}.
   \]
The identity at \eqref{mod4ideq1} follows upon inserting this Bailey
pair in \eqref{Seq1a}, and using the Jacobi Triple Product identity
to sum pairs of series.
\end{proof}

\subsection{Theta-False Theta ``hybrid" identities.}  \label{hybrid}
Rogers~\cite{R17} referred to a series of the form
\begin{equation} \label{ThetaSeriesDef}
\sum_{n=-\infty}^\infty (\pm 1)^n q^{r n^2 + s n}
= \sum_{n=0}^\infty (\pm 1)^n q^{r n^2 + s n}
+ \sum_{n=1}^\infty (\pm 1)^n q^{r n^2 - s n}
\end{equation}
as a \emph{theta series of order $r$}, where $r \in \mathbb{Q}_+$ and
$s\in\mathbb Q$.
Due to the Jacobi triple product identity, the series~\eqref{ThetaSeriesDef}
may be expressed as the infinite product
\[  \left(  \mp q^s, \mp q^{2r-s}, q^{2r}; q^{2r} \right)_\infty. \]

  Rogers also took interest in series of the form
\begin{equation} \label{FalseThetaSeriesDef}
\sum_{n=0}^\infty (\pm 1)^n q^{r  n^2 + s n}
- \sum_{n=1}^\infty (\pm 1)^n q^{r n^2 -s n}
= \sum_{n=0}^\infty (\pm 1)^n q^{r n^2 + s n} (1 - q^{2(r-s)n + (r-s) } )
\end{equation}
which he called a \emph{theta series of order $r$}.  False
theta series are not representable as infinite products, but nonetheless
 may have representations as Rogers-Ramanujan type $q$-series,
and Rogers gave a number of examples of identities of false theta series.
See~\cite[\S5, p. 35 ff]{McLSZ08} for further discussion and numerous
examples of false theta series identities.

Here we give
several examples of identities where one side is a basic
hypergeometric series and the other side comprises a theta product
multiplied by a false theta series. We believe this type of identity
is new.

\begin{corollary}\label{cor3.10}
\begin{equation}\label{Seq1aqfeq3}
\sum_{n=0}^{\infty}\frac{(-q;q^2)_{n}q^{n^2}}{(1-q^{2n+1})(q^2,q^2;q^2)_{n}}
=\frac{(-q;q^2)_{\infty}}{(q^2;q^2)_{\infty}}
\left( 1 + \sum_{r=0}^\infty q^{8r^2 + 4r+1} (1- q^{8r+4}) \right).
\end{equation}
\begin{equation}\label{misc1eq3}
\sum_{n=0}^{\infty}\frac{q^{n^2+n}}{(1-q^{2n+1})(q,q;q)_n}
= \frac{ \sum_{r=0}^\infty q^{6r^2+2r} (1-q^{8r+4})}{(q;q)_\infty}.
\end{equation}
\begin{equation}\label{misc1eq6}
\sum_{n=0}^{\infty}\frac{q^{n^2+2n}}{(1-q^{2n+1})(q,q;q)_n}
=\frac{\sum_{r=0}^\infty q^{6r^2+4r} (1-q^{4r+2})}{(q;q)_\infty}.
\end{equation}
\begin{equation}\label{misc1eq4}
\sum_{n=0}^{\infty}\frac{(-q;q)_nq^{(n^2+n)/2}}{(1-q^{2n+1})(q,q;q)_n}
=\frac{(-q;q)_{\infty}}{(q;q)_{\infty}}
\sum_{r=0}^\infty q^{4r^2+r} (1-q^{6r+3}).
\end{equation}
\begin{equation}\label{misc1eq8}
\sum_{n=0}^{\infty}\frac{(-q;q)_nq^{(n^2+3n)/2}}{(1-q^{2n+1})(q,q;q)_n}
=\frac{(-q;q)_{\infty}}{(q;q)_{\infty}}
\sum_{r=0}^\infty q^{4r^2 + 3r} (1-q^{2r+1}).
\end{equation}
\begin{equation}\label{misc1eq7}
\sum_{n=0}^{\infty}\frac{q^{n^2+2n}}{(1-q^{2n+1})(q^2;q^2)_n}
=(-q;q^2)_{\infty}
\sum_{r=0}^\infty (-1)^r q^{6r^2+4r} (1-q^{4r+2}).
\end{equation}

\begin{equation}\label{misc1eq10}
\sum_{n=0}^{\infty}\frac{(-q;q^2)_{n+1}
q^{n^2-2n}}{(q^2;q^2)_n(q^2;q^2)_{n+1}}\\
=\frac{(-q;q^2)_{\infty}}{(q^2;q^2)_{\infty}}\left [
1+q+q^2+q^3+\sum_{r=0}^{\infty}q^{r^2}(q^{-2r}-q^{2r})\right ].
\end{equation}
\end{corollary}

\begin{proof}
The identity at \eqref{Seq1aqfeq3} follows  upon setting $d= q^2$
and  $a=q$ in the pair at \eqref{sgen661}, substituting the
resulting pair (Slater's pair \textbf{I(10)}) into \eqref{Seq1aqf}
(with $a=q$), replacing $q$ with $q^2$, and finally $q$ with $-q$.

The identity at \eqref{misc1eq3} follows from setting $d=q^2$ and
$a=q$ in the Bailey pair at \eqref{sgen661}, and substituting the
resulting pair into \eqref{Seq1a}, while the identity at
\eqref{misc1eq4} follows from substituting the same pair in
\eqref{Seq1b}, and replacing $q$ with $q^{1/2}$.

Likewise, the identity at \eqref{misc1eq6} follows as a consequence
setting $d=q^2$ and $a=q$ in the Bailey pair at \eqref{sgen662}, and
substituting the resulting pair into \eqref{Seq1a}, while for
\eqref{misc1eq7} we proceed similarly, after creating a Bailey pair
by instead setting $a=-q$ (and keeping $d=q^2$).

The identity at \eqref{misc1eq8}  follows upon inserting the pair
used in the proof of \eqref{misc1eq6} into \eqref{Seq1b}, after
replacing $q$ with $q^{1/2}$.

For \eqref{misc1eq10}, insert Slater's pair \textbf{H(12)} ($a=q$,
$c=q$ and $b \to 0$ in \eqref{Sgen1}) into \eqref{Seq1aqf} (with
$a=q$), replace $q$ with $q^2$, and rearrange.

\end{proof}

\section{Concluding Remarks}
In a paper by the second author~\cite{S07}, it was shown that
more than half of Slater's identities could be derived from
just three general Bailey pairs together with several limiting
cases of Bailey's lemma and an associated family of
$q$-difference equations.  Here, we attempt to put Slater's
work in a broader context via general Bailey pairs, but
without the use of $q$-difference equations.
Both approaches have their merits, and both yield new identities.
It may well be worth exploring the combinatorial consequences of
the new identities presented here.

We also note that we have given just a sample of the identities that may be produced by the methods used in Corollaries \ref{cor3.6} - \ref{cor3.10}, and that it is likely that many similar identities may be produced by employing other Bailey pairs.

\section{Addendum: May 2023}
On May 8, 2023, the second author received an email from Aritram Dhar sharing the following
observations:

\begin{quotation}
I was going through your 2010 paper ``Some implications of Chu's ${}_{10}\psi_{10}$ extension of Bailey's 
${}_{6}\psi_6$ 
summation formula'' joint with J. McLaughlin and P. Zimmer and I came across the 
${}_{10}\psi_{10}$  identity in Corollary 1.2 which follows by replacing $b$ with 
$a/c$ in Proposition 1.1 which is the 
${}_{10}\psi_{10}$  identity in Theorem 2 of Chu's 2006 paper 
``Bailey's very well-poised ${}_{6}\psi_{6}$-series identity''.  
I have a comment regarding your Corollary 1.2 which is as follows:

If you look at Exercise 5.26 on page 152 of Gasper and Rahman's Basic Hypergeometric Series text, 
the identity there is a ${}_{6+2k}\psi_{6+2k}$ series-product identity. 
As mentioned at the end of the problem, if 
we refer to Chu's 1998 paper ``Partial-fraction expansions and well-poised bilateral series'' 
and look at 
Theorem 2 on page 500, then it is essentially the same identity as that in Gasper and Rahman. 
Now, 
if we consider $(k,n_1,n_2,N) = (2,1,1,2)$ and make the substitution 
$(a,b,c,d,e_1,e_2)\to (a,c,d,e,qu,qv)$ in 
Gasper and Rahman's Exercise 5.26 ${}_{6+2k}\psi_{6+2k}$ identity, we get exactly Corollary 1.2 of your paper. 
By making a similar substitution in the ${}_{6+2k}\psi_{6+2k}$  identity, I believe we can get Proposition 1.1 (Theorem 2 of Chu's 2006 paper) of your paper.   
\end{quotation}

\section{Acknowledgment}
We thank Aritram Dhar for his careful reading of our paper, and for sharing his observations 
now included in Section 5.

 \allowdisplaybreaks{

}

\end{document}